\documentclass[11pt,reqno]{amsart}

\usepackage{amsmath,amsthm,amssymb,enumerate}
\usepackage{color}

\theoremstyle{plain}

\newtheorem{theorem}{Theorem}[section]

\newtheorem{corollary}[theorem]{Corollary}
\newtheorem{proposition}[theorem]{Proposition}

\theoremstyle{definition}

\newtheorem{definition}[theorem]{Definition}
\newtheorem{remark}[theorem]{Remark}
\newtheorem{example}[theorem]{Example}

\numberwithin{equation}{section}

\newcommand{\sgn}{\mathop{\mathrm{sgn}}\nolimits}

\newcommand{\commen}[1]{}

\newcommand{\C}{\mathbb{C}}

\newcommand{\R}{\mathbb{R}}

\begin{document}

\title{Carleman factorization of layer potentials on smooth domains}

\author[Ando Kang Miyanishi Putinar]{Kazunori Ando, Hyeonbae Kang, Yoshihisa Miyanishi and Mihai Putinar}

\address{Department of Electrical and Electronic Engineering and Computer Science, Ehime University, Ehime 790-8577, Japan, {\tt ando@cs.ehime-u.ac.jp}}

\address{Department of Mathematics, Inha University, Incheon
402-751, S. Korea, {\tt hbkang@inha.ac.kr}}

\address{University of California at Santa Barbara, CA 93106, USA and
Newcastle University, Newcastle upon Tyne NE1 7RU, UK, {\tt mputinar@math.ucsb.edu, mihai.putinar@ncl.ac.uk}}

\address{   Department of Mathematical Sciences,
   Faculty of Science, Shinshu University,
   A519, Asahi 3-1-1, Matsumoto 390-8621, Japan
{\tt  miyanishi@math.shinshu-u.ac.jp}}

\thanks{KA and YM were partially supported by JSPS of Japan KAKENHI grants 20K03655 and 21K13805, HK by NRF (of S. Korea) grant 2021R1A2B5B02-001786, and MP by a Simons Collaboration Grant for Mathematicians.}


\maketitle

\begin{abstract} One of the unexplored benefits of studying layer potentials on smooth, closed hypersurfaces of Euclidean space is the factorization of the Neumann-Poincar\'e operator into a product of two self-adjoint transforms. Resurrecting some pertinent indications of Carleman and M. G. Krein, we exploit this grossly overlooked structure by confining the spectral analysis of the Neumann-Poincar\'e operator to the amenable $L^2$-space setting, rather than bouncing back and forth the computations between Sobolev spaces of negative or positive fractional order. An enhanced, fresh new look at symmetrizable linear transforms enters into the picture in the company of geometric-microlocal analysis techniques. The outcome is manyfold, complementing recent advances on the theory of layer potentials, in the smooth boundary setting.

\end{abstract}
\newpage

\tableofcontents

\section{Introduction} Let $\Omega \subset \R^d, \  d\geq 2,$ be a bounded domain with smooth boundary $\Gamma$. The single layer $S$, respectively double layer potential $K^\ast$,
are compact, singular integral operators acting on Lebesgue space $L^2(\Gamma)$ (we recall the precise definitions in the preliminaries below). The foundational works of Carl Neumann (on convex domains) and Poincar\'e (on smooth domains) reduce the solvability of Dirichlet problem to the spectral decomposition of the operator $K^\ast$. Much of the XX-th Century theory of integral equations and early spectral analysis is rooted in this very specific question. Nowadays $K^\ast$ is called the Neumann-Poincar\'e operator; its qualitative analysis was resurrected a couple of decades ago by practitioners of applied field theory. To the extent that today we count the recent references to layer potentials in the thousands.

An immutable complication of this approach is the non symmetry of the integral kernels of $K^\ast$ and its adjoint $K$. While the single layer potential is self-adjoint $S = S^\ast$, and positive $S>0$ after a minor rescaling of $\Omega$ (in two dimensions), the double layer potential operator is only {\it symmetrizable} in the weaker, Hilbert space norm of the fractional Sobolev space $H^{-1/2}(\Gamma).$ This explains why the spectrum of $K^\ast$ is real, with each non-zero eigenvalue of finite multiplicity, and no generalized eigenvectors (spectral Jordan blocks). These basic attributes are encoded and derivable from the simple intertwining identity
\begin{equation}\label{Plemelj}
S K^\ast = K S
\end{equation}
attributed to Plemelj \cite{Plemelj}. The fascination with the Hilbert space geometry framework streaming from Plemelj's identity goes deep, spanning more than a century, with contributions and rediscoveries due to several generations of mathematicians. The leading question being how pathological is the spectral behavior of a compact {\it symmetrizable operator} when compared with the Hilbert-Schmidt theory of symmetric operators. A section of the present article, containing both old and new results, offers a fresh look at symmetrizable operators from the perspective of the abstract theory of non-selfadjoint operators.

The case of a bounded domain with smooth boundary $\Gamma \subset \R^d$ is reinforced by a much stronger algebraic feature of the double layer potential transform:
\begin{equation}\label{fac}
K^\ast = A S,
\end{equation}
where $A$ is a self-adjoint, bounded operator on $L^2(\Gamma)$. This validates Plemelj indentity: $S K^\ast = SAS = K S$ and opens the perspective of doing spectral analysis
of $K^\ast$ (\`a la Hilbert and Schmidt) in the original Lebesgue space, without invoking a rather intricate Sobolev space of negative order. The factorization (\ref{fac}) is obtained below
by basic pseudo-differential calculus arguments. Incidentally, the concept of pseudo-differential operator was coined by Friedrichs and Lax in a historical issue of the journal Communications in Pure and Applied Mathematics \cite{FL-65}, with a first ever application to the symmetrization of a hyperbolic system of partial differential equations. In dimension two, of a much earlier date, the boundedness of the factor $A$ is mentioned tangentially by Carleman in his 1916 doctoral dissertation: pg. 159 in \cite{Carleman}. At the abstract level of Hilbert space symmetrizable operators, the importance of the identity (\ref{factor}) was recognized by M. G. Krein \cite{Krein}. Products of self-adjoint operators continue to this day to be investigated in themselves \cite{RW, CDMM}.

The observation that the double layer operator associated to a smooth hypersurface is a product of two self-adjoint operators has notable ramifications. First, the spectral resolution
\begin{equation}\label{res}
K^\ast = \sum_{j=0}^\infty \lambda_j \langle \cdot, g_j \rangle f_j
\end{equation}
converges in operator norm. Here $\lambda_j$ are non-zero eigenvalues of $K^\ast$ with an associated  biorthogonal system of normalized eigenvectors
$$K^\ast f_j = \lambda_j f_j, \ \
K g_j = \lambda_j g_j, \ \ \langle f_j, g_\ell\rangle = \delta_{j\ell}, \ j, \ell \geq 0.$$
As a matter of fact, the convergence of the series (\ref{res}) is even stronger, uniform from $L^2(\Gamma)$ to $H^s(\Gamma)$ with $s<1$ in all dimensions and arbitrary $s>0$
for $d=2$ (true analogs of Mercer's Theorem). Second, a $C^{(1)}(\R)$ functional calculus, continuous in the operator norm, for  $K^\ast$ exists. This gives a good grasp on the resolvent and generalized Fredholm determinant of the double layer potential.  For instance, the resolvent expression
$$ \frac{ (I-zK^\ast)^{-1} - I}{z} =  \sum_{j=0}^\infty \frac{ \langle \cdot, g_j \rangle f_j}{\frac{1}{\lambda_j} - z}$$
is given by a  Borel series of simple fractions, uniformly convergent  on compact subsets of $\C \setminus \sigma(K^\ast)$, with respect to the operator topology norm. Moreover, the resolvent of
the Neumann-Poincar\'e operator exhibits a growth surprisingly close to that of a self-adjoint operator:
 $$\frac{1}{|z - \alpha|} \leq \| (K^\ast - z)^{-1} \| \leq \frac{ 1+ \| A \| \|S \|}{|z-\alpha|},$$
 where $\alpha$ is a non-zero eigenvalue, and $z$ is sufficiently close to $\alpha$.
We also show that $K^\ast$ allows spectral synthesis, that is its range is hereditary complete, with an effective
identification of certain cyclic vectors, when they exist.

The factor $A$ in (\ref{factor}) is a psedo-differential operator of order zero, entering by a simple identity
$$ 2 \Lambda =  S^{-1} +  A,$$
into the structure of the meticulously investigated Dirichlet to Neumann map $\Lambda$. The principal symbol of $A$ has a closed form expression, involving differential geometric invariants of $\Gamma$.
We note that the commutator $[A,S] = K^\ast - K$ vanishes if and only if $\Gamma$ is a sphere. Moreover, we prove that the operator $A$ is a stable observable under the
geodesic flow on the cosphere bundle of $\Gamma$.

Traditionally, the solution to Dirichlet's problem is reduced to an integral equation on the boundary involving solely the double layer potential. Knowing the factorization
$K^\ast = A S$ on a bounded, smooth domain $\Omega \subset \R^d, \ d \geq 2,$ offers a closed form expression
$$ u = -S_{Af} - 2 D_f, \ \ f \in L^2(\Gamma),$$
of the solution, a.k.a.  Poisson's transform: $ \Delta u = 0$ in $\Omega$ and $u|_\Gamma = f$ (in a weak sense).
\bigskip

The contents is the following. Section \ref{preliminaries} collects a series of known results of Newtonian potential theory, spectral analysis of integral operators and microlocal analysis.
Section \ref{symmetrizable} elaborates at the level of current terminology and accumulated results the abstract theory of symmetrizable linear operators. Specific observations pertaining to the spectral synthesis of symmetrizable operators are collected in Section \ref{synthesis}. Section \ref{analysisNP} is devoted to novel aspects of the spectral analysis of the Neumann-Poincar\'e operator derived from the general theory of symmetrizable operators combined with recent geometric analysis advances.

\section{Preliminaries}\label{preliminaries}

We recall in this section some terminology and basic facts of Newtonian potential theory, spectral analysis and pseudodifferential calculus.
We warn the reader that there is no consensus in the vast literature on the subject of signs and constants in the definitions of potentials. We hope this will not be a cause of confusion.

\subsection{Potentials and jump formulae}

Let $d \geq 2$ be an integer and $\Omega$ be a bounded domain in $\mathbf R^d$ with boundary $\Gamma$.
For the time being we assume that $\Gamma$ is at least $C^2$-smooth. The $(d-1)$-dimensional surface measure on $\Gamma$
is $d\sigma$ and the unit outer normal to a point $y \in \Gamma$
will be denoted by $n_y$.

We let $E(x,y)=E(x-y)$ stand for the normalized
Newtonian kernel:
\begin{equation}
E(x,y) = \left\{ \begin{array}{cc}
           \frac{1}{2\pi} \log \frac{1}{|x-y|}, & d=2,\\
           c_d |x-y|^{2-d}, & d \geq 3,
         \end{array}\right.
\end{equation}
where $c_d^{-1}$ is the surface area of the unit sphere in
$\mathbf R^d$. The signs were chosen so that $\Delta E = -\delta$
(Dirac's delta-function).

We associate to a $C^2$-smooth function (density, in physical terms) $f(x)$ on $\Gamma$ the fundamental potentials:
the {\it single and double layer potentials} in $\mathbf R^d$; denoted by $S_f$ and $D_f$ respectively:
\begin{equation}
S_f(x) = \int_\Gamma E(x,y) f(y) d\sigma(y),\ \ \
D_f(x) = \int_\Gamma \frac{\partial E}{\partial n_y} (x,y) f(y) d \sigma(y).
\end{equation}
The surface $\Gamma$ divides $\mathbf R^d$ into two domains $\Omega = \Omega_i$
(interior to $\Gamma$) and the exterior $\Omega_e$. Thus the potentials above
define pairs of functions $(S_f^i, S_f^e)$ and $(D_f^i,D_f^e)$ which are harmonic
in $\Omega_i$ and $\Omega_e$ respectively.

As is well known from classical potential theory (cf. \cite{Kellog})
denoting by $S_f^i(x), \frac{\partial}{\partial n_x} S_f^i (x)$ (and corresponding symbols with superscript e) the
limits at $x \in \Gamma$
from the interior (or exterior), the following relations (known as the jump formulas for the
potentials) hold for all $x \in \Gamma$:
\begin{align}
\begin{split}
S_f^i(x) &= S_f^e(x);\\
\textstyle\frac{\partial}{\partial n_x} S_f^i(x) &= \textstyle\frac{1}{2} f(x) + \textstyle\int_\Gamma  \frac{\partial E}{\partial n_x}
(x,y) f(y) d\sigma(y);\\
 \textstyle\frac{\partial}{\partial n_x} D_f^i(x) &= \textstyle \frac{\partial}{\partial n_x} D_f^e(x);\\
  \textstyle D_f^i(x) &= -\textstyle\frac{1}{2} f(x) +\textstyle\int_\Gamma  \frac{\partial E}{\partial n_y} (x,y) f(y) d\sigma(y);\\
 \textstyle\frac{\partial}{\partial n_x} S_f^e(x) &=  -\textstyle\frac{1}{2} f(x) + \textstyle\int_\Gamma  \frac{\partial E}{\partial n_x}
(x,y) f(y) d \sigma(y);\\
D_f^e(x) &= \textstyle\frac{1}{2} f(x) + \textstyle\int_\Gamma  \frac{\partial E}{\partial n_y} (x,y) f(y) d\sigma(y).
\end{split}
\end{align}

Rather direct computations (see for instance
 \cite{Mazya}) show that
the integral kernels
$$ K(x,y):=  -\frac{\partial}{\partial n_y} E(x-y); \ \  K^\ast(x,y) =  -\frac{\partial}{\partial n_x} E(x-y),$$
satisfy on $\Gamma$ the following estimates, for $d \geq 3$:
\begin{align}
\begin{split}
|K(x,y)| = O(\textstyle\frac{1}{|x-y|^{d-2}}), \ \ x,y \in \Gamma, x \neq y,\\
|K^\ast(x,y)| = O(\textstyle\frac{1}{|x-y|^{d-2}}), \ \ x,y \in \Gamma, x
\neq y.
\end{split}
\end{align}

For $d=2$, due to the fact that $\log|z-w|$ is the real part of
a complex analytic function $\log(z-w) = \log|z-w| + i \arg(z-w), \ z,w \in \Gamma$, and by Cauchy-Riemann's
equations one obtains
$$ K(z,w) = \frac{\partial}{\partial \tau_w} \arg(z-w),$$
where $\tau_w$ is the unit tangent vector to the curve $\Gamma$. Thus, on any
smooth curve $\Gamma \subset \mathbf R^2$, the kernels $K(z,w)$ and $K^\ast(z,w)$ are  uniformly bounded.

Returning to the general $d$-dimensional case, we define on $L^2(\Gamma) = L^2(\Gamma, d\sigma)$ the
bounded integral transform:
\begin{equation}
(Kf)(x) = 2 \int_\Gamma K(x,y) f(y) d\sigma(y), \ \ f \in L^2(\Gamma, d\sigma).
\end{equation}
The $L^2$ adjoint $K^\ast$ will be an integral operator with
kernel $K^\ast(x,y)$, traditionally called the {\it Neumann-Poincar\'e  operator}.  The nature of the diagonal singularity of
the kernel $K(x,y)$ shows  that $K$ is a compact operator in the
Schatten-von Neumann class $\mathcal{C}^p(L^2(\Gamma)), p > d-1,$
see \cite{Kellog}. Since the kernel $K$ is bounded when $d=2$, it
is Hilbert-Schmidt on any smooth planar curve. We will show in the
next section that $K^\ast$ is symmetrizable, that is $K^\ast$
becomes self-adjoint with respect to a different (incomplete)
inner product on $L^2 (\Gamma)$.

Similarly, the linear operator $$ Sf = S_f|_\Gamma, \ \ f \in
L^2(\Gamma),$$ turns out to be bounded (from $L^2(\Gamma)$ to the
same space). Remark that the representing kernel $E(x,y)$ of $S$
is pointwise  non-negative for $d\geq 3$. With these conventions
the jump formulas become, as functions on $\Gamma$:
\begin{align}\begin{split}\label{jump}
  {S_f}^i &= S_f^e = Sf;\\
 \partial_n S_f^i &= \textstyle\frac{1}{2} f - \textstyle\frac{1}{2} K^\ast f;\\
\partial_n S_f^e &= -\textstyle\frac{1}{2} f - \textstyle\frac{1}{2} K^\ast f;\\
D_f^i &= -\textstyle\frac{1}{2} f - \textstyle\frac{1}{2}Kf;\\
D_f^e &= \textstyle\frac{1}{2} f - \textstyle\frac{1}{2}Kf.\\
\end{split}
\end{align}
Above, and always in this paper $n$ designates the {\it outer} normal
to $\Omega$.

\subsection{The Dirichlet to Neumann map}

Let $\Omega$ be a bounded domain with smooth boundary $\Gamma$ in $\R^d, d\geq 2.$
Given a smooth function $f \in L^{2}(\Gamma)$, the {\it Dirichlet to Neumann map}
$\Lambda f$ is the normal derivative on $\Gamma$, of the solution to the Dirichlet problem
$$ \Delta u = 0 \ \ {\rm in}\ \Omega, \ \ u|_\Gamma = f.$$
Specifically,
$$ \Lambda f = \partial_n u \ \ {\rm on} \ \Gamma.$$
In view of the jump formulae satisfied by the layer potential kernels (the second line of \eqref{jump}), we find
\begin{equation}\label{eq: relationship between Dirichlet-Neumann and layer potentials}
\Lambda S =  \frac{1}{2}  -\frac{1}{2} K^\ast.
\end{equation}

It turns out that the Dirichlet to Neumann map, as well as the layer potential integrals are pseudodifferential
operators (abbreviated to $\Psi$DO from now on) acting on the smooth, compact manifold $\Gamma$. Their (principal) symbols are at hand, offering an
effective tool towards the spectral analysis of $\Lambda, K, S$. Details on the Dirichlet to Neumann map can be
found in the articles \cite{GP,GKLP} and the book \cite{LMP}.

\subsection{Pseudo-differential calculus}\label{PDO}

We recall below a few results of the theory of pseudo-differential operators, trimmed at some specific computations involving layer potentials.

If $d\geq 3$ and the boundary surface $\Gamma$ is smooth, then the operators $\Lambda,\ S,\ K$ can be interpreted as $\Psi$DOs of
order $1, -1, -1$ respectively \cite{GP, GKLP, Ola}. Their principal symbols are respectively:
\begin{align*}
\sigma_0(\Lambda) &= |\xi|_x :=\sum_{i, j=1}^{d-1} \sqrt{g^{ij} \xi_i \xi_j}, \\
\sigma_0(S) &=\frac{1}{2} |\xi|_x^{-1} := \frac{1}{2}\left[\sum_{i, j=1}^{d-1} \sqrt{g^{ij} \xi_i \xi_j} \right]^{-1}, \\
\sigma_0(K) &=\sigma_0(K^\ast)  =  |\xi|_x^{-3} \left[ \sum_{j=1}^{d-1} \kappa_{j}(x)  |\xi|_x^2 -L(\xi, \xi) \right],
\end{align*}
where $g^{ij}$ is the metric tensor of the real hypersurface $\Gamma$ embedded in Euclidean space $\R^d$, $\xi =(\xi_1, \xi_2, \cdots, \xi_{d-1})$ denotes a vector belonging to the cotangent bundle
$T^\ast_x(\Gamma)$, $\kappa_j(x)$ are the principal curvatures of $\Gamma$ at $x$, and $L(\xi, \xi)$ is the second fundamental form. See for instance \cite{Ola} for details.

In virtue of the general principles of pseudo-differential calculus, the operator $A =S^{-1} K = K^{\ast} S^{-1} =S^{-1} - 2 \Lambda $ is a $\Psi$DO of order 0 whose principal symbol is given by the closed form expression
\begin{equation}\label{symbolA}
\sigma_0(A)(x, \xi) = |\xi|_x^{-2} \left[ \sum_{j=1}^{d-1} \kappa_{j}(x)  |\xi|_x^2 -L(\xi, \xi) \right], \ \ d \geq 3.
\end{equation}

The situation $d=2$ is even simpler. Indeed, assuming the curve $\Gamma$ is smooth, the operator $S^{-1}$ is a $\Psi$DO of order one, while
$K^\ast$ is a smoothing $\Psi$DO \cite{FKM}. Hence $A = K^\ast S^{-1}$ is also smoothing.
The operator $A$ is bounded on the Sobolev space $H^{s}(\Gamma)$, for each $s \in {\mathbf{R}}$. If $d=2$, then
$A : H^s(\Gamma) \longrightarrow {\mathcal C}^{(\infty)}(\Gamma)$ is linear and continuous. We will resume the analysis and consequences of the factorization $K^\ast = AS$ in the last section of this article.

\subsection{Integral operators and their kernels}

Not unrelated to solving Dirichlet problem via Carl Neumann boundary integral equation method, the
discoveries of Hilbert and Schmidt offer even today a solid reference landmark. We recall a few details
of this classical chapter of functional analysis, aimed in the present article to provide a comparison term for a
more general and less circulated framework (of symmetrizable operators). A basic reference for Hilbert-Schmidt theory we rely on is Riesz and Nagy
monograph \cite{RN}.

Let $\mu$ be a finite, positive Borel measure, defined on a compact space $X$. On Lebesgue space $L^2(X,\mu)$ one considers a linear operator $T_K$ with square integrable kernel:
$$ (T_K f)(x) = \int_X K(x,y) f(y) d\mu(y), $$
that is
$$ \int |K(x,y)|^2 d\mu(x) d\mu(y) < \infty.$$
Assuming $K$ is symmetric,
$$ K(x,y) = \overline{K(y,x)}, \ \ x,y \in X,$$
the spectrum of $T_K$ is real, with only possible accumulating point at zero.
The spectral decomposition of the compact operator
\begin{equation}
T_K f = \sum_{n=0}^\infty \lambda_n \langle f, f_n \rangle f_n,\ \ f \in L^2(X,\mu),
\end{equation}
converges in  $L^2(X,\mu), $uniformly with respect to $f$. Above $f_n$ is the normalized eigenfunction corresponding to the eigenvalue $\lambda_n$;
$$ T_k f_n = \lambda_n f_n, \ \ \| f_n \| =1.$$

On the other hand, the integral kernel itself admits the convergence
$$ K(x,y) = \sum_n \lambda_n f_n(x) f_n(y),$$
this time in the norm topology of $L^2(X \times X, \mu \otimes \mu).$

A stronger integrability condition imposed on the integral kernel enhances the convergences of the spectral resolution of $T_K$.
The typical example is a continuous kernel $K \in C( K \times K).$
Assume for instance,
\begin{equation}\label{uniform}
 \int_X |K(x,y)|^2 d\mu(y) \leq C^2, \ \ x \in X.
 \end{equation}
For a sequence $\varphi_n \rightarrow \varphi$ in $L^2$ one finds the uniform, pointwise estimate:
$$ |(T_K \varphi_n)(x) - (T_K \varphi)(x)|^2 \leq |\int K(x,y) (\varphi_n - \varphi)(y) d\mu(y)|^2 \leq C^2 \| \varphi_n - \varphi \|_2^2.$$
In particular, taking $\varphi_n = \sum_{k=0}^n \langle \varphi, f_k \rangle f_k$,
$$ (T_K \varphi)(x) = \sum_k \lambda_k \langle \varphi, f_k\rangle f_k(x)$$
converges uniformly, with respect to $x \in X$.

In the above scenarios, the abstract solution  $u = (\lambda - T_K)^{-1}f$ of
the integral equation:
$$ \lambda u(x) - \int K(x,y) u(y) d\mu(y) = f(x),\ \ x \in X, \ \ f \in  L^2(X,\mu)$$
 is given by a convergent series, with uniform bounds with respect to $f$, or even $x$:
$$ u(x) = \sum \frac{\langle f, f_n\rangle}{\lambda - \lambda_n} f_n(x), \ \ \lambda \notin \sigma(T_K).$$
We refer for details to Section 97 in \cite{RN}.

A relaxed condition on the kernel, similar to the one invoked above:
$$ \int_X |K(x,y)|^2 d\mu(y) < \infty, \ \ \mu-a.e. \ \ x \in X,$$
was proposed by Carleman, involving this time unbounded, linear operators. See for instance the survey \cite{Akhiezer}.

Among all improving convergence results, {\it Mercer's Theorem}  stands aside for simplicity and
versality. It states that, if the real valued,  integral kernel $K(x,y)$ is symmetric and continuous on $X$ compact, and $T_K \geq 0$, or equivalently all
eigenvalues are non-negative: $\lambda_k \geq 0$, then the eigenfunctions
$f_n(x)$ are continuous for all $\lambda_n >0$, and moreover the series expansion
$$ K(x,y) = \sum_n \lambda_n f_n(x) f_n(y)$$
is converging {\it uniformly and absolutely}  in $X \times X$. The monograph \cite{Krasnoselskii} contains a detailed analysis from a unifying perspective of Mercer's Theorem.

%{\color{red}It is not so easy to catch the point of \S 2 ?}

\section{Symmetrizable operators}\label{symmetrizable}

The basics of the spectral theory of symmetrizable operators have been reassessed and rediscovered by many authors. Unquestionably the origin of the concept of symmetrizable operator is the Plemelj intertwining condition (\ref{Plemelj}) satisfied by the layer potentials. With the exception of Carleman \cite{Carleman}, who builds with authority the abstract theory of symmetrizable operators with the specific aim at treating the spectral resolution of the Neumann-Poincare integral operator, all authors who independently claim priority on the same topics do not mention potential theory as a motivation. We list in chronological order, spanning a good century only a part of these (re)discoverers: Marty (1910) \cite{Marty}, Mercer (1920) \cite{Mercer}, M. G. Krein (1937-47) \cite{Krein}, Zaanen (1946-53) \cite{Za}, Reid (1949) \cite{Reid}, Wielandt (1950) \cite{Wielandt}, Lax (1956) \cite{Lax}, Dieudonn\'e (1961) \cite{Dieudonne}, Veic (1962) \cite{Veic}, Sebest\'en and Tarcsay (2011) \cite{ST}. While the principal statements are the same, the proofs proposed by them are different, with notable variations complementing each other. A common trend to all authors is the quest for similarities and deviations between the spectral analysis of symmetrizable operators and the Hilbert-Schmidt theory of compact self-adjoint operators.

Leaving aside the topics of layer potentials, the symmetrization technique of specific linear operators found spectacular applications to the analysis of Cauchy's problem for hyperbolic systems of
partial differential equations. The groundbreaking article by Friedrichs \cite{Friedrichs-54} amply illustrates that the symmetry of the principal symbol of a first order, linear hyperbolic system of partial differential operators is paramount for applying the general theory of Hilbert space towards establishing the existence of (weak) solutions to the Cauchy problem. One of the first lucid applications of the emerging theory (at that time) of $\Psi$DOs is due to Friedrichs and Lax \cite{FL-65}; there, Friedrichs 1954 framework is expanded beyond self-adjoint, matrix valued principal symbols by the symmetrization procedure already exploited for generations in the spectral analysis of integral operators. Nowadays, the so-called {\it hyperbolic symmetrizer method} is one of the canonical tools appearing in the study of an array of Cauchy problems associated to hyperbolic systems of partial differential equations (with rough initial data,  non-smooth coefficients, non-linear, etc.) \cite{GR, Jannelli, JT}.

We derive below the main results pertaining to symmetrizable operators closely following the perspective outlined in Section 11 of Carleman's doctoral dissertation \cite{Carleman},
independently and masterly complemented by M. G. Krein around 1937, an article to become available to the western readers only in 1998 \cite{Krein}! Some of the results contained in this section are new.

\subsection{Spectral analysis}
Let $H$ be an infinite dimensional, complex, separable Hilbert space. We denote by ${\mathcal L}(H)$ the algebra of linear, bounded operators acting on $H$. Let $S \in {\mathcal L}(H)$ be a positive self-adjoint operator:
$$ \langle Sx, x\rangle > 0, \ \ x \neq 0.$$
The opertator $K \in  {\mathcal L}(H)$ is called {\it symmetrizable} (by $S$) if
$$ SK^\ast = KS.$$
Without loss of generality we assume $\| S \| =1.$
In other terms, $K^\ast$ is symmetric with respect to the inner product norm defined by $S$:
$$ \langle S K^\ast x, y\rangle = \langle Sx, K^\ast y \rangle, \ \ x,y \in H.$$
We denote by $\sqrt{S}$ the positive square root of the operator $S$.
The only interesting case is when $S$ is not invertible, that is $\| x \|^2_S = \langle Sx, x \rangle$ is a weaker, non-equivalent Hilbert space norm:
$$ \| x \|^2_{-1} = \| \sqrt{S} x \|^2 \leq \|x \|^2, \ \ x \in H.$$

Let $H_{-1}$ be the Hilbert space completion of $H$ with respect to the weaker norm $\| \cdot \|_{-1}$.
The first non-trivial observation (Lemma 1 in \cite{Krein}) is that $K^\ast$ extends by continuity to a linear and bounded operator on $H_{-1}$. That is, there exists a positive constant $M$
with the property
$$ \| \sqrt{S} K^\ast x \| \leq M \| \sqrt{S} x \|, \ \ x \in H.$$
Consequently there exists an operator $C \in {\mathcal L}(H)$ satisfying
$$ \sqrt{S}K^\ast = C \sqrt{S}.$$
Then
$$ S K^\ast = \sqrt{S} C \sqrt{S} = K S$$ is a self-adjoint operator, implying via the density of the range of $S$ that $C = C^\ast$ automatically.
Moreover, the intertwining
$$ K \sqrt{S} = \sqrt{S} C$$
shows that the operator $K$ leaves invariant the subspace $H_1 = \sqrt{S} H.$ Note that $\sqrt{S} H$ is a non-closed vector subspace of $H$.
With these conventions in place, $H_1$ is a Hilbert space endowed with the norm
$$ \| x \|_{1} = \| \sqrt{S}^{-1}x \|, \ \ x \in H_1= \sqrt{S} H.$$
We recognize here a so called Gelfand triple:
$$ H_{1} \subset H \subset H_{-1},$$
that is two dense inclusions of Hilbert spaces, with a non-degenerate pairing
$$ H_1 \times H_{-1} \longrightarrow \C, \ \ (x,y) \mapsto \langle x, y \rangle, \ \ x \in H_1, y \in H_{-1},$$
induced by the middle inner product. Within this framework we can also regard $K$ as a linear bounded operator acting on $H_1$ and
$K^\ast$ linear and bounded in $H_{-1}.$ This time, both operators $K \in {\mathcal L}(H_1)$ and $K^\ast \in {\mathcal L}(H_{-1})$ are self-adjoint.
By the {\it spectral subspace} associated to an eigenvalue $\lambda$ we denote the collection of all generalized eigenvectors associated to $\lambda$, that is solutions $f$ of
$ (T-\lambda)^n f = 0,$ where $n$ can be a natural number.
\begin{theorem}\label{thm: Theorem3 in Krein} (Theorem 3 in \cite{Krein}) \label{compact} Assume $K \neq 0$ is a compact symmetrizable operator. Then
both operators $K \in {\mathcal L}(H_1)$ and $K^\ast \in {\mathcal L}(H_{-1})$ are compact.

The spectrum $\sigma(K)$ of $K$ is real with non-zero eigenvalues.
For every $\lambda \in \sigma(K) \setminus \{0\}$ the associated spectral subspaces of $K^\ast \in {\mathcal L}(H)$ and  $K^\ast \in {\mathcal L}(H_{-1})$ coincide,
while $S$ maps bijectively the spectral subspace of $K^\ast \in {\mathcal L}(H)$ onto the spectral subspace of $K \in {\mathcal L}(H).$
\end{theorem}

With this result at hand, the parallel to Hilbert-Schmidt expansion of a compact self-adjoint operator is becoming more transparent. Let $K$ be compact and symmetrizable as in Theorem \ref{thm: Theorem3 in Krein}.
The only interesting case is when the spectrum of $K$ infinite, with zero as an accumulation point:
$$ \sigma (K) = \{ 0\} \cup \{ \lambda_0, \lambda_1, \lambda_2, \ldots \}.$$ We do not exclude in the above enumeration $\lambda _k = \lambda_\ell$ for different values of the indices, so that
a system of eigenvectors can be chosen:
$$ K^\ast f_n = \lambda_n f_n,  \ \ n \geq 0,$$
$$ K g_n = \lambda_n g_n, \ \ n \geq 0.$$
Since the eigenvectors $f_n$ are diagonalizing the selfadjoint operator $K^\ast \in {\mathcal L}(H_{-1})$, we adopt the normalization
$$ \langle \sqrt{S} f_k, \sqrt{S} f_\ell \rangle = \delta_{k\ell}.$$
Note that $g_n = Sf_n, \ \ n \geq 0$ and also
$$C \sqrt{S} f_n = \lambda_n \sqrt{S}f_n, \ \ n \geq 0.$$

We claim that these are all eigenvectors of $C$, corresponding to non-zero eigenvalues. Indeed, assume
$$ C\phi = \mu \phi.$$ Then
$$ K \sqrt{S}\phi = \sqrt{S} C \phi = \mu \sqrt{S} \phi.$$
In view of Theorem \ref{thm: Theorem3 in Krein} we have listed all eigenvectors of $K$. Hence there exists $n$, so that
$\mu = \lambda_n$ and $\phi$ is proportional to $g_n = S f_n$. We can take then $\phi = \sqrt{S}f_n.$

In conclusion, the operator $C$ is compact and has the orthogonal spectral decomposition
\begin{equation}\label{resolC}
C x = \sum_{n=0}^\infty \lambda_n \langle x, \sqrt{S} f_n\rangle \sqrt{S} f_n.
\end{equation}
Moreover, the above convergence is uniform with respect to $x$, that is
$$ C =  \sum_{n=0}^\infty \lambda_n \langle \cdot , \sqrt{S} f_n\rangle \sqrt{S} f_n$$
in operator norm.

From the above expansion we deduce weak spectral resolutions  (in $H$) for both operators $K$ and $K^\ast$.
Specifically
\begin{equation}\label{weak-K*}
 \sqrt{S} K^\ast x = \sum_{n=0}^\infty \lambda_n \langle x, g_n\rangle \sqrt{S} f_n, \ \ x \in H,
 \end{equation}
and
\begin{equation}\label{weak-K}
 K \sqrt{S} x = \sum_{n=0}^\infty \lambda_n \langle x, \sqrt{S} f_n \rangle g_n, \ \ x \in H.
 \end{equation}

Removing the factor $\sqrt{S}$ in the above operator norm expansions is desirable, but not automatic.
Bridging the gap between symmetrizable and symmetric operators, we note first a spanning property of the eigenvectors in the first scenario.

\begin{corollary}\label{span}
In the conditions of Theorem \ref{thm: Theorem3 in Krein}, the eigenvectors, including the null vectors, of $K$ span $H$.

Moreover, the eigenvectors corresponding to the non-zero eigenvalues span respectively the closed ranges of $K$ and $K^\ast$.
\end{corollary}

\begin{proof} We prove first that the eigenvectors $g_n$ (corresponding to non-zero eigenvalues) span the closed range of $K$.
Assume $h \in \overline{ {\rm Ran} (K)} \ominus {\rm} {\rm span} \{ g_n, \ n \geq 0 \}$.
Then $Sh \perp {\rm span} \{ f_n, \ n \geq 0\}$. But
$$S K^\ast h =  \sum_{n=0}^\infty \lambda_n \langle h, g_n\rangle S f_n = 0,$$
that is $K^\ast h = 0$, or equivalently $h \perp  \overline{ {\rm Ran} (K)}.$ That is $h=0$.

Second, assume $f \in \overline{ {\rm Ran} (K^\ast)} \ominus {\rm} {\rm span} \{ f_n, \ n \geq 0 \}$.
Then $Sf \in  \overline{ {\rm Ran} (K)} \ominus  {\rm span} \{ g_n, \ n \geq 0 \}$, hence $Sf =0$, and $f=0$.

The intertwining relations $\sqrt{S} K^\ast = C \sqrt{S}$ and $ \sqrt{S} C = K \sqrt{S} $ prove the inclusions $\sqrt{S} \ker K^\ast \subset \ker C$
and $\sqrt{S} \ker C \subset \ker K.$ In addition, $\sqrt{S} C H \subset K H$. But $C$ is a self-adjoint operator, that is
$\ker C \oplus C H$ is a dense subspace of $H$.  The operator $\sqrt{S}$ has also dense range, hence $\ker K + K H $ is a dense subspace of $H$.
\end{proof}

\begin{example}\label{non-completeness}
Next we provide an example of a symmetrizable operator $K$ with the property that $\ker K^\ast + K^\ast H$ is not dense in the underlying  Hilbert space $H$.
To this aim, consider an orthonormal basis $e_0, e_1, e_2, \ldots$ of $H$ and a decreasing sequence $(\lambda_n)_{n=0}^\infty$ of positive real numbers satisfying
$$ \lambda_0 =1, \ \  \sum_{n=1}^\infty \lambda_n^2 = 1.$$
Let $S = {\rm diag} (\lambda_0, \lambda_1, \lambda_2 , \ldots)$ and
$$ A = \left[ \begin{array}{ccccc}
     1 & -\lambda_1 &  -\lambda_2 & -\lambda_3 & \ldots\\
     -\lambda_1 & 1 & 0 & 0 & \\
     -\lambda_2 & 0 & 1 & 0 & \ldots \\
     -\lambda_3 & 0 & 0 & 1 & \\
     \vdots & 0 & 0 & & \ddots  \\
     \end{array}\right].$$
  A vector $(x_0, x_1, x_2, \ldots)$ belongs to the kernel of the self-adjoint operator $A$ if and only if
  $$ x_j = \lambda_j  x_0, \ \ j \geq 1.$$ Hence $\dim \text{Ker} A = 1$.   On the other hand the operator
  $$ AS =   \left[ \begin{array}{ccccc}
     1 & -\lambda^2_1 &  -\lambda^2_2 & -\lambda^2_3 & \ldots\\
     -\lambda_1 & \lambda_1 & 0 & 0 & \\
     -\lambda_2 & 0 & \lambda_2 & 0 & \ldots \\
     -\lambda_3 & 0 & 0 & \lambda_3 & \\
     \vdots & 0 & 0 & & \ddots  \\
     \end{array}\right]$$
   has a trivial kernel.  Indeed, if  $(y_0, y_1, y_2, \ldots) \in \ker (AS)$, then
   all entries $y_j = y_0, \ j \geq 0,$ are equal.

In conclusion, the symmetrizable operator $K = SA$ has a non-trivial kernel, while its adjoint
$K^\ast = AS$ is injective. That means that the range of $K^\ast$ is not dense in $H$.
For completeness we display the matrix associated to the self-adjoint factor $C$, entering into the structure of $K$, namely
$C = \sqrt{S} A \sqrt{S}$:
 $$ C =   \left[ \begin{array}{ccccc}
     1 & -\lambda^{3/2}_1 &  -\lambda^{3/2}_2 & -\lambda^{3/2}_3 & \ldots\\
     -\lambda^{3/2}_1 & \lambda_1 & 0 & 0 & \\
     -\lambda^{3/2}_2 & 0 & \lambda_2 & 0 & \ldots \\
     -\lambda^{3/2}_3 & 0 & 0 & \lambda_3 & \\
     \vdots & 0 & 0 & & \ddots  \\
     \end{array}\right].$$
The vector  $(x_0, x_1, x_2, \ldots)$ belongs to the kernel of $C$ if and only if
$ x_j = \sqrt{\lambda_j} x_0, \ \ j \geq 1.$ So, in this example both cases $\ker C = 0$
or $\dim \ker C =1$ can occur, while $\ker K$ is always non-trivial.
   \end{example}

The first example of a compact operator $L$ with a complete system of eigenvectors with its adjoint $L^\ast$ not possessing a complete system of eigenvectors
was found in 1951 by Hamburger \cite{Hamburger}. For a discussion of such pathologies see \cite{Markus,Nikolskii}.

\begin{theorem} \label{strong-conv} Assume, in the conditions of Theorem \ref{compact}, that $K(H) \subset H_1= \sqrt{S}(H).$ Then
 $$ K^\ast x = \sum_{n=0}^\infty \lambda_n \langle x, g_n\rangle f_n, \ \ x \in H,$$
 and
 $$ K x = \sum_{n=0}^\infty \lambda_n \langle x,  f_n \rangle g_n, \ \ x \in H.$$
 \end{theorem}

 Both Carleman \cite{Carleman} and Krein \cite{Krein}  have reached this conclusion, by imposing similar, or even a stronger range condition. For completeness we include a proof.

 \begin{proof} Let $x , y\in H$ with the property $Kx = \sqrt{S}y$. Then
 $$ \langle Kx, f_n\rangle = \langle y, \sqrt{S}f_n \rangle, \ \ n \geq 0.$$ Since $(\sqrt{S}f_n)$ is an orthonormal system
 $$ \tau = \sum_n | \langle Kx, f_n \rangle|^2 = \sum_n |\langle x, f_n \rangle|^2 \lambda_n^2 < \infty.$$

  According to a Lemma explicitly stated by Gelfand \cite{Gelfand} (in itself emerging from the works of Orlicz \cite{Orlicz}), there exists a constant $M>0$, with the property
 $$ \sum_n | \langle Kx, f_n \rangle|^2  \leq M^2 \| x \|^2, \ \ x \in H.$$
 More precisely, the lower semi-continuity of the seminorm
 $$ [\sum_n | \langle Kx, f_n \rangle|^2]^{1/2}$$
 and its finiteness on the entire Hilbert space assure the uniform bound above. Consequently, the coefficients of the orthogonal decomposition are square summable, and
 $$ \|  \sum_n \lambda_n \langle x , f_n \rangle \sqrt{S} f_n \| \leq M \|x \|.$$
 Moreover, we find by applying the contraction  $\sqrt{S}$:
 $$  \|  \sum_n \lambda_n \langle x , f_n \rangle S f_n \| \leq M  \|x\|.$$
 On the other hand,
 $$ K x = \sum_n \lambda_n \langle x, f_n \rangle g_n, \ \ x \in \sqrt{S}(H).$$
 The density of the range of $\sqrt{S}$ implies the second identity in the statement.

 To prove the expansion of $K^\ast$, we consider the partial sum
 $$ K_N = \sum_{n=0}^N \lambda_n \langle \cdot, f_n \rangle g_n.$$
 Since for every vector $x \in H$, we have
 $$ \lim_N K_N(x) = K(x), $$
 Banach-Steinhaus Theorem implies
 $$ \sup_N \| K_N \|  = \sup \| K_N^\ast\| < \infty.$$
 In addition,
 $$K_N^\ast x = \sum_{n=0}^N \lambda_n \langle x, g_n \rangle f_n,$$
 and $\lim K^\ast_N x = K^\ast x$ for $x$ in the dense subspace formed by finite sums of $f_n's$.
 The uniform bound of the norms of $K_N^\ast$ completes the proof.

 \end{proof}

 \subsection{Factorization of a symmetrizable operator}\label{factor}

The existence of the factorization $K = S A$ with $A \in {\mathcal L}(H)$ is a guarantee for the applicability of Theorem \ref{strong-conv}.
Almost all contributors to the theory of symmetrizable operators recognize the importance of this algebraic condition which, in general, is far from being fulfilled.
Note that, in this fortunate case
$$  KS = SAS = SK^\ast,$$
therefore $A$ is a self-adjoint operator. In addition, from $\sqrt{S} K^\ast = C \sqrt{S}$ one finds
$$ C = \sqrt{S} A \sqrt{S}.$$ Of interest for the applications to Neumann-Poincar\'e operator is the compactness of $S$.

\begin{corollary}\label{norm-conv} Let $K = S A$ be a compact symmetrizable operator, with $S$ compact and $A$ bounded.
Then the expansions
$$ K =  \sum_{n=0}^\infty \lambda_n \langle \cdot,  f_n \rangle g_n,$$
and
$$ K^\ast =   \sum_{n=0}^\infty \lambda_n \langle \cdot,  g_n \rangle f_n,$$
converge in operator norm.
\end{corollary}

Only now, on rather restrictive ground, we have reached a true analog of the Hilbert-Schmidt expansion of a self-adjoint, compact operator.

\begin{proof} The spectral decomposition of the compact self-adjoint operator $S$ implies that $\sqrt{S}$ is also compact.
For any vector $x \in H$,
\begin{align*}
K^\ast x &=  \sum_{n=0}^\infty \lambda_n \langle x,  g_n \rangle f_n \\
&=\sum_{n=0}^\infty  \langle \sqrt{S}x,  \sqrt{S}f_n \rangle K^\ast f_n \\
&=\sum_{n=0}^\infty  \langle \sqrt{S}x,  \sqrt{S}f_n \rangle A \sqrt{S}  \sqrt{S}f_n \\
&=A \sqrt{S} \sum_{n=0}^\infty  \langle \sqrt{S}x,  \sqrt{S}f_n \rangle \sqrt{S} f_n.
\end{align*}
The compact operator $A \sqrt{S}$ transforms a strong operator topology convergent sequence into a norm convergent one.

The second norm convergent expansion follows by taking adjoints.
\end{proof}

In the conditions of Corollary \ref{norm-conv}, the solution of the integral equation
$$ \lambda u - K^\ast u = f, \ \ f \in H, \ \lambda \notin \sigma(K),$$
is given by a uniformly (with respect to $f$) convergent series
\begin{equation}\label{resolvent-series}
 u = \frac{f}{\lambda} + \sum_{n=0}^\infty \frac{\lambda_n}{\lambda( \lambda - \lambda_n)} \langle f, g_n\rangle f_n.
 \end{equation}

We analyze in a subsequent section the possibility of improving the norm convergence of this resolvent series.

Obviously, the truncations
$$ K_N = \sum_{k=0}^N \lambda_k \langle \cdot, f_k \rangle g_k$$ provide finite rank {\it symmetrizable} approximations of the operator $K$.
The rate of convergence is however indirect. For a better grasp on the possible finite rank approximation of $K$ we propose the following scheme,
relaxing a bit the notion of finite rank symmetrizable operator.

\begin{proposition}\label{finite-rank}
Assume $K = S A$ is a symmetrizable operator, with $S$ compact and $A$ self-adjoint and bounded.
Let
$$ S = \sum_{k=0}^\infty \sigma_k \langle \cdot, \phi_k \rangle \phi_k,$$
where $(\phi_k)$ is an orthonormal basis of eigenvectors of $S$ and $\sigma_k \geq \sigma_{k+1}, \ \ k \geq 0,$ satisfies $\lim_k \sigma_k = 0.$
Denote by $\pi_N$ the orthogonal projection onto \\
${\rm span} \{ \phi_0, \phi_1, \ldots, \phi_{N-1}\}$ and $S_N = \pi_N S \pi_N$.

Then $L_N =  S_N A, \ \ N \geq 1,$ are finite rank, $S_N$-symmetrizable operators, and
$$ \| K - L_N \| \leq \|A \| \sigma_N, \ \ N \geq 1.$$

\end{proposition}

\begin{proof} The subspace generated by the eigenvectors in invariant under $S$, hence $S_N = \pi_N S = \pi_N S \pi_N = S \pi_N$.
Then $L_N = \pi_N S \pi_N A $ and the estimate in the statement follows from $\| S - S_N \| \leq \sigma_N.$
\end{proof}

The more desirable finite central approximation $\pi_N K \pi_N = S (\pi_N A \pi_N)$, is $S$-symmetrizable at every step. Its rate of convergence would involve a quasi-diagonal decay of the factor $A$, such as
$ \lim_N \| (I-\pi_N) A \pi_N \| = 0.$

The case $A = I + B$, where $B$ is a compact operator is notable, in particular for its links to Keldysh perturbation theory techniques. The monograph \cite{GK} contains an accessible overview of this important chapter of operator theory. We reproduce below an illustrative situation.

\begin{theorem}\label{thm: asymptotics decreasing singular values} Let $K = (I+B) S$ be a symmetrizable operator with $S>0$ compact and $B=B^\ast$ compact. Assume $K^\ast$ is injective. Then the singular numbers of $K$ and $S$ enumarating in descending order are asymptotically equal:
$$ \lim_{k \rightarrow \infty} \frac{\sigma_k(K)}{\sigma_k(S)} = 1.$$
\end{theorem}

If $K^\ast$ is injective, then $I+B$ is injective, and by the stability of the Fredholm index under compact perturbations, $I+B$ is invertible. For details see Theorem V.11.2 in \cite{GK}.
A generalization of the above classical result will be presented in the last section, in the context of the spectral analysis of the Neumann-Poincar\'e operator.
In the conditions of Theorem \ref{thm: asymptotics decreasing singular values} , a similar asymptotic equivalence holds, this time for the characteristic (counting) functions $n(r, \cdot)$ of the eigenvalues of $K$ and $S$, assuming that a nondecreasing function $\phi(r), r \geq 0$, exists, subject to the constraints:
$$ \frac{\phi(s)}{\phi(r)} \leq \frac{s^\gamma}{r^\gamma}, $$
for a positive constant $\gamma$ and sufficiently large $r<s$, and second,
$$ \lim_{r \rightarrow \infty} \frac{n(r,S)}{\phi(r)} = 1.$$
Then
$$ \lim_{r \rightarrow \infty} \frac{n(r,K)}{n(r,S)} = 1.$$
We refer to Theorem V.11.1 in \cite{GK} for details. We will resume in the last section the discussion of spectral asymptotics comparison.

\subsection{Functional calculus with smooth functions}

We adopt the conditions imposed in the statement of Corollary \ref{norm-conv}. Let $(\tau_n)_{n=0}^\infty$ be a bounded sequence of complex numbers.
Let $x \in H$ and consider the series
$$ \sum_n \lambda_n \tau_n \langle x, g_n \rangle f_n = A\sqrt{S} [ \sum_n \tau_n \langle \sqrt{S} x, \sqrt{S} f_n \rangle \sqrt{S} f_n ].$$
Moreover, the orthogonal series in the parantheses is convergent, with uniform norm with respect to $x$:
$$ \| \sum_n \tau_n \langle \sqrt{S} x, \sqrt{S} f_n \rangle \sqrt{S} f_n \|^2 = \sum_n |\tau_n|^2 |\langle \sqrt{S} x, \sqrt{S} f_n \rangle|^2 \leq M \| x \|^2,$$
where $M = \sup_n |\tau_n|^2.$ Since $\sqrt{S}$ is a compact operator we infer the convergence in operator norm of the series
\begin{equation}\label{eq: convergence of finite rank approximated op} \sum_n \lambda_n \tau_n \langle \cdot , g_n \rangle f_n,\end{equation}
the limit being of course a compact operator.

\begin{theorem}\label{funct-calculus}
Let $K$ be a compact symmetrizable operator of the form $K= SA$, with $S>0$ compact. Let $[a,b]$ be a finite, closed interval which contains the spectrum of $K$.
There exists a norm continuous, unital algebra homomorphism
$$ \Phi: C^{(1)}[a,b] \longrightarrow {\mathcal L}(H),$$
given by the formula
$$ \Phi (\phi) = \phi(0) I + \sum_n [\phi(\lambda_n) - \phi(0)] \langle \cdot, g_n \rangle f_n.$$
If $p(x)$ is a polynomial function, then $\Phi(p) = p(K^\ast).$
\end{theorem}

\begin{proof} Let $x \in H$ and $\phi \in C^{(1)}[a,b] $. The series defining $\Phi(\phi)$ converges by the remark above \eqref{eq: convergence of finite rank approximated op} . Moreover
$$ \|  \Phi(\phi)x \| \leq 2 \| A \| \| \phi \|_{C^{(1)}} \|x \|,$$
or equaivalently
$$ \| \Phi(\phi)\| \leq 2 \| A \| \| \phi \|_{C^{(1)}}.$$
Therefore $\Phi$ is a linear, continuous map.

According to Corollary \ref{norm-conv}, for any positive integer $m$ we find
$$ (K^\ast)^m = \sum_n \lambda_n^m \langle \cdot, g_n \rangle f_n.$$
Then on a polynomial $p(x) = p(0) + x q(x)$,  the map $\Phi$ acts as follows:
$$ \Phi(p) = p(0) I + \sum_n q(\lambda_n) \lambda_n \langle \cdot, g_n \rangle f_n = p(0) I+ K^\ast q(K^\ast) = p(K^\ast).$$
The density of polynomials in $C^{(1)}[a,b]$ completes the proof of the multiplicativity of $\Phi$.

\end{proof}

\begin{corollary} Let $\phi \in C^{(1)}[a,b] $ be a real valued function vanishing at $x=0$. Then the operator $\Phi(\phi)$ is compact, symmetrizable,
with spectrum $\sigma(\Phi(\phi)) = \phi \sigma(K^\ast).$

If the function $\phi$ does not vanish at the origin, then $\Phi(\phi)$ is a symmetrizable operator, compact perturbation of $\phi(0) I$.
\end{corollary}

\begin{proof}
The compactness follows from the norm convergence of the series defining $\Phi$, while the intertwining
$$ S \Phi(K^\ast) = \Phi(K^\ast)^\ast S$$
follows by polynomial approximation.
\end{proof}

The analogy to the continuous functional calculus of self-adjoint operators is noticeable. As a matter of fact, the operator $\Phi(\phi)$ only depends on the
values of $\phi$ on the spectrum of $\sigma(K^\ast)$. A lower bound for the above functional calculus follows from the following key inequality proved by M. Krein.

\begin{theorem}(Theorem 1 in \cite{Krein}) Let $K \in {\mathcal L}(H)$ be a symmetrizable operator with respect to $S>0$: $ SK^\ast = KS$. Then
\begin{equation}\label{ineq}
\sup_{x \neq 0}  \frac{ \langle S K^\ast x, x \rangle}{\langle S x, x\rangle} \leq \| K^\ast \|.
\end{equation}
\end{theorem}

An independent proof was later discovered by Lax, see the proof of Theorem I in \cite{Lax}.

\begin{corollary} In the conditions of Theorem \ref{funct-calculus} one finds for every function $\phi \in C^{(1)}[a,b] $:
\begin{equation}\label{lower-bound}
\| \phi \|_{\infty, \sigma(K)} \leq \| \Phi(\phi) \|.
\end{equation}
\end{corollary}

\begin{proof} In view of the estimate (\ref{ineq}) applied to the operator $\Phi(\phi)$, one derives $\Phi(\phi) = \phi(K^\ast)$, this time
interpreted as a genuine functional calculus of the symmetric transform $K^\ast$, acting on the space $H_{-}$ (the closure of $H$
with respect to the weaker norm $\langle Sx,x\rangle$). The spectral theorem for self-adjoint operators yields
$$ \| \phi(K^\ast) \|_{{\mathcal L} (H_{-})} = \|  \phi \|_{\infty, \sigma(K)} .$$

\end{proof}

One can immediately exploit this inequality for deriving resolvent growth estimates at points of the spectrum. We elaborate the details in a subsequent section devoted to the Neumann-Poincar\'e operator.

\subsection{Improved spectral resolution} In an attempt to find a proper analog of Mercer's Theorem for symmetrizable operators, Krein has
advocated the following setting, streaming on an stronger regularity of the operator $K^\ast$. The setting is the same as in the previous sections:
$KS = SK^\ast$, with $S>0$.

Let $X$ be a Banach space, densely contained in $H$, endowed with a stronger norm than that of $H$:
$$ \| x \| \leq \| x \|_X, \ \ x \in X.$$
We assume $K^\ast (X) \subset X$ and impose the strong continuity condition:
\begin{equation}\label{strong-cont}
 \| K^\ast x \|_X \leq \gamma \| \sqrt{S}x \|, \ \ x \in X.
 \end{equation}

 Then the extension of $K^\ast$ to the weak space $H_{-1}$ satisfies $K^\ast (H_{-1}) \subset X$. In particular, the eigenfunctions $f_n$ of $K^\ast$ belong to $X$.
 Indeed, let $u \in H_{-1}$ and consider a sequence of elements $ u_n \in H$ converging to $u$ in $H_{-1}$. In view of the strong continuity imposed on $K^\ast$, we find
 $$ \| K^\ast u_j - K^\ast u_n \|_X \leq \gamma \| u_j - u_k \|_{H_{-1}}.$$
 Hence $K^\ast u_n$ is a Cauchy sequence in $X$, converging to an element $x \in X$.  Since the topology of $H_{-1}$ is weaker, and $K^\ast \in {\mathcal L}(H_{-1})$,
 we find $x = K^\ast y$.

 In particular, the eigenfunctions $f_n$ of $K^\ast$ belong to $X$.

\begin{theorem}\label{Krein} (Theorem 4 in \cite{Krein}) Let $K \in {\mathcal L}(H)$ be a compact symmetrizable operator subject to the continuity assumption \ref{strong-cont},
where $X$ is a Banach space densely contained in $H$. Then the spectral resolution
$$ K^\ast x = \sum_n \lambda_n \langle x, g_n \rangle f_n, \ \ x \in X,$$
converges in the norm $\| \cdot \|_X$.
\end{theorem}

\begin{proof}
Let $L \in X^\ast$ be a linear continuous functional. Assumption (\ref{strong-conv}) and Riesz Lemma imply the existence of an element $y_L \in H_{-1}$ satisfying
$$ L(K^\ast x) = \langle Sx, y_L \rangle, \ \ x \in X.$$
In particular, the sequence
$$ L(K^\ast f_n) = \lambda_n L(f_n) = \lambda_n \langle S f_n, y_L\rangle$$ is square summable, because the system of vectors $(f_n)$ is orthonormal in $H_{-1}$.
Therefore
$$ \sum_n \lambda_n^2 |L(f_n)|^2 < \infty.$$
Gelfand's lemma \cite{Gelfand} implies the existence of a constant $\delta>0$ with the property
$$  \sum_n \lambda_n^2 |L(f_n)|^2 \leq \delta^2 \| L \|_{X^\ast}.$$
Towards proving that the partial sums of the spectral resolution of $K^\ast$ form a Cauchy sequence in $X$, we choose
integers $p <q$ and consider the finite section:
$$ s_{pq} = \sum_{n=p}^q \lambda_n \langle x, g_n\rangle f_n.$$
Hahn-Banach Theorem yields a linear functional $L_{pq}$, of norm $1$ and satisfying
$$ L_{pq}(s_{pq}) = \| s_{pq}\|_B.$$
Then
$$ \| s_{pq} \|_B = L_{pq} [ \sum_{n=p}^q \lambda_n \langle x, g_n\rangle f_n] = $$ $$
\sum_p^q \langle x, g_n\rangle L_{pq} (K^\ast f_n) \leq $$ $$\sqrt{ \sum_p^q |\langle x, g_n\rangle|^2} \sqrt{ \sum \lambda_n^2 |L_{pq}(f_n)|^2}
\leq \delta \sqrt{ \sum_p^q |\langle x, g_n\rangle|^2}$$
and the latter can be made less than a prescribed $\epsilon$, as soon as $q > p \geq N_\epsilon$.
\end{proof}

Regardless to say that the resolvent series  (\ref{resolvent-series}) converges then in the stronger norm of the space $X$.

\subsection{Biorthogonal systems of vectors in Hilbert space}

Above we have started with a compact symmetrizable operator and explored its spectral resolution in terms of the eigenfunctions of it and its adjoint.
The picture can be reversed, starting with a biorthogonal system of vectors in Hilbert space and producing from there a compact symmetrizable operator.
This point of view was adopted by Veic \cite{Veic}. We merely sketch the main construction.

Let $H$ be a complex, separable Hilbert space of infinite dimension.
A collection of vectors $(f_n)_{n=0}^\infty$ is a {\it complete, minimal system of vectors}
if it spans $H$ and
$$ f_k \notin {\rm span} \{ f_j, \ j \neq k\}, \ \ k \geq 0.$$
Then a biorthogonal dual system of vectors exists: $(g_n)_{n=0}^\infty$:
$$ \langle f_k, g_\ell \rangle = \delta_{k\ell}.$$
In general, the Fourier type expansion associated to an element $f \in H$
$$ f \sim \sum_{n=0}^\infty \langle f, g_n \rangle f_n,$$
may not converge.

A minimal constraint is to assume that $(g_n)_{n=0}^\infty$ is called a {\it Bessel system}. That is
$$ \sum |\langle f, g_n \rangle|^2 < \infty, \ \ f \in H.$$
In that case Gelfand's Lemma \cite{Gelfand} implies:
$$  \sum |\langle f, g_n \rangle|^2 \leq M^2 \|f \|^2.$$
The dual array of vectors $(g_n)$ is a {\it Hilbert system}, that is:
for every $(\alpha_n) \in \ell^2$ there exists $f \in H$ such that
$$ \langle f, f_k\rangle = \alpha_k, \ \ k \geq 0.$$
In this case there is a constant $m >0$:
$$ m^2 \| f \|^2 \leq \sum |\langle f, f_k\rangle|^2, \ \ f \in H.$$
Consequently
$$ \| g_n \| \leq \frac{1}{m}, \ \ \|f_n\| \geq m, \ \ n \geq 0.$$

From the above minimal setting, the weak norm (of the space $H_{-1})$ pops-up.

\begin{theorem}(Bari \cite{Bari})
A Bessel system of vectors $(f_n)$ and its biorthogonal dual
$(g_n)$ are related by a positive, linear bounded operator $S >0$:
$$ g_n = S f_n, \ \ n \geq 0.$$
\end{theorem}

We recognize by now that $(\sqrt{S} f_n)$ is an orthonormal basis of $H$.
The restricted convergence of the skew Fourier expansion emerges:
$$ \sum |\langle f, g_n\rangle|^2 = \sum |\langle f, S f_n\rangle|^2 = $$ $$ \sum \langle |\sqrt{S} f, \sqrt{S} f_n \rangle |^2 \leq \| \sqrt{S} f \|^2 \leq M^2 \| f \|^2, \ \ f \in H.$$
On the other hand
$$ \sum |\langle f, f_n \rangle|^2 < \infty$$ if and only if $f \in {\rm Ran} \sqrt{S}.$

Indeed, if $f = \sqrt{S}h$, then
$$ \sum |\langle f, f_n\rangle|^2 = \sum |\langle h, \sqrt{S} f_n\rangle|^2 < \infty.$$
Conversely, let $(e_n = \sqrt{S}f_n )$ be the associated  orthonormal system of $H$, so that
$$ h = \sum \langle f, f_n \rangle e_n$$ converges in $H$.
Then
$$ \langle f, f_n \rangle = \langle h, \sqrt{S} f_n \rangle = \langle \sqrt{S}h , f_n\rangle, \ \ n \geq 0.$$

Choose a sequence of numbers $(\lambda_n)$ converging to zero and define a linear map $K^\ast$
on the algebraic span $\mathcal F$ of the vectors $(f_n)$ by
$$ K^\ast f_n = \lambda_n f_n, \ \ n \geq 0,$$
It turns out that $K^\ast$ is a symmetric bounded operator on $H_{-1}$. Then necessarily $\lambda_n$ are all real.
Define $K$ on the linear span $\mathcal G$ of the vectors $(g_n)$ by
$$ K g_n = \lambda_n g_n.$$
The map $S: H_{-1} \longrightarrow H_1$ is unitary, hence $K \in {\mathcal L}(H_1)$ is also a bounded, symmetric operator. 
Note that $\mathcal G$ is a subset of $H_1 = S (H_{-1})$.

We find
$$ \langle K^\ast f, g \rangle = \langle f, K g \rangle, \ \ f \in {\mathcal F}, \ g \in {\mathcal G}.$$
In addition, the self-adjoint operator $C$ defined by
$$ C \sqrt{S} f_n = \lambda_n \sqrt{S} f_n, \ \ n \geq 0,$$
shows that  $\sqrt{S} K^\ast = C \sqrt{S}$ extends by continuity from $\mathcal F$ to a compact operator defined on $H$.
Moreover
$$ S K^\ast = \sqrt{S} C \sqrt{S} = K S$$ is self-adjoint and compact on $H$.
The expansions
$$ \sqrt{S} K^\ast f = \sum \lambda_n \langle f, g_n\rangle \sqrt{S} f_n, \ f \in H,$$
and
$$ K f = \sum \lambda_n \langle f, f_n \rangle g_n, \ \ f \in {\rm Ran} \sqrt{S},$$
remain valid without having in hand the boundedness of $K$, or $K^\ast$, as linear transforms of $H$.

\section{Spectral synthesis}\label{synthesis}

Spectral synthesis of a collection of operators, usually a group or a semigroup, encodes the spanning (``synthesis") of joint invariant subspaces by the common (generalized) eigenvectors contained there.
The concept is highly relevant in harmonic analysis on groups, or homogeneous spaces, where the structure of translation invariant subspaces is a central topics. Even for a single operator, the spectral synthesis of its invariant subspaces is highly relevant, and not trivial. The surveys \cite{Markus,Nikolskii} offer a glimpse on the subject, while the recent expository note
\cite{Baranov} guides the reader through the most recent advances.

\begin{definition}
Let $T \in {\mathcal L}(H)$ be a linear bounded operator with a complete system of eigenvectors
$$ (T-\lambda_i) f_i = 0, \  i \in I; \ \ {\rm span} \{ f_i, \ i \in I\} = H.$$
The operator $T$ allows {\it spectral synthesis}  if every $T$-invariant, closed subspace of $T$ is generated by a subset of eigenvectors.
\end{definition}

A self-adjoint operator with a complete system of eigenvectors admits spectral synthesis. However,  normal operators may not have this property.
Wermer \cite{Wermer} proved the following striking characterizations. Let $N$ be a normal operator with a complete system of eigenvectors. The following are equivalent:
\bigskip

{\bf 1.}  $N$ admits spectral synthesis,
\bigskip

{\bf 2.} Every $N$-invaraint subspace contains an eigenvector of $N$,
\bigskip

{\bf 3.} Every $N$-invaraint subspace is also $N^\ast$-invariant,
\bigskip

{\bf 4.} The point spectrum $\sigma_p(N)$ does not carry a nonzero measure, orthogonal to all polynomials:
if $$ \sum_{\lambda \in \sigma_p(N)}  \frac{\mu(\lambda)}{z-\lambda} = 0, \ \ |z| >>1,$$ then $\mu=0$.

Compact operators also hide a variety of pathologies. As already remarked in the previous section, we owe to Hamburger  an example of a compact operator with a complete system of eigenvectors
such that the adjoint does not possess a complete system of eigenvectors. Ane even worse,  N. K. Nikolskii produced the following example.
Let $V$ be a Volterra operator, that is $V$ compact and $\sigma(V) = \{ 0 \}.$ There exists a compact operator $K$
with simple spectrum and complete system of eigenvectors, so that $V = K|_L$, where $L$ is a $K$-invariant subspace. See for details \cite{Nikolskii}.

Symmetrizable operators satisfy a variant of spectral synthesis, as stated below.

\begin{theorem}
Let $K \in {\mathcal L}(H)$ be a compact symmetrizable operator: $KS = SK^\ast$, with $S>0$.
Let $P$ denote the projection of  $H$ onto a closed  invariant subspace of $K^\ast$. Then $PK^\ast P$ is
$PSP$-symmetrizable, the eigenvectors of $PKP$ generate $P H$ and the eigenvectors of $PK^\ast P$ generate its closed range
$\overline{ K^\ast P H}$.
\end{theorem}

\begin{proof}

The invariance of the subspace $P H$ under $K^\ast$ implies:
$$ K^\ast P = P K^\ast P, \ \ P K = PKP.$$
Then
$$ PSP PK^\ast P = P S K^\ast P = P K S P = PKP PSP,$$
and $PSP>0$ as an operator on $PH$. In other terms,
$PK^\ast P$ is a compact $PSP$-symmetrizable operator.
According to Corollary \ref{span}, the eigenvectors, including the null vectors, of $PK P$,
span the subspace $P(H).$

\end{proof}

As Example \ref{non-completeness} shows, the only complication preventing the restricted operator $PK^\ast P$ to have a complete system of eigenvectors is the possible non-trivial kernel of the
operator $P K = P K P : PH \longrightarrow PH.$

Traditionally, sufficient conditions for spectral synthesis were obtained via Phragmen-Lindel\"of theorem, by
imposing decay conditions on the singular numbers of the respective operator. Section 2 of Markus' survey \cite{Markus}
contains relevant statements with ample references. We extract from there a single application to symmetrizable operators.

\begin{proposition} Let $S > 0$ be a positive self-adjoint operator belonging to a Schatten-von Neumann class $C_p$ with $p < \infty$.
Let $A$ be a compact self-adjoint operator with $I+A$ injective. Then both the symmetrizable operator $K = S (I+A)$ and its adjoint
$K^\ast$ allow spectral synthesis.
\end{proposition}

The proof exploits bounds of the resolvent $(I - \lambda K)^{-1}$ on wedges with vertices at $\lambda =0$ which avoid the real axis:
$ 0 < \epsilon < \arg \lambda < \pi - \epsilon.$
For details we  refer to Theorem 2.1 in \cite{Markus}.

The typical example of an invariant subspace for $K^\ast$ is the cyclic subspace generated by a single vector $\xi$:
$$ L = {\rm span} \{ (K^\ast)^n \xi, \ \ n \geq 0 \}.$$
An eigenvector $f_n$ corresponding to a non null simple eigenvalue $\lambda_n$ of $K^\ast$ belongs to $L$  only if
$\langle \xi, g_n\rangle  \neq 0.$ Indeed, if $f_n \in L$, then there exists a polynomial $p(x)$ with the property that $p(K^\ast)\xi$
is arbitrarily close to $f_n$, that is
$$ \langle  p(K^\ast)\xi, g_n \rangle \neq 0,$$
or equivalently
$$ p(\lambda_n) \langle \xi, g_n \rangle \neq 0.$$
Choosing $p(\lambda_n) \neq 0$, we derive $\langle \xi, g_n\rangle  \neq 0.$

To prove the converse statement we need to impose on the spectral expansion of $K^\ast$ some stronger convergence conditions.
Moreover, working solely with cyclic subspaces we can assume that all eigenvalues are simple. The most natural framework is stated below.

\begin{proposition}\label{cyclic} Let $K = S A \in {\mathcal L}(H)$ be a symmetrizable operator, with
$A$ self-adjoint and $S>0$ compact. Assume all non-zero eigenvalues $\lambda_j, \ j \geq 0,$ of $K$ are simple
and $K$ is injective. A vector $\xi \in H$ is $K^\ast$-cyclic if and only if
$$ \langle \xi, g_j \rangle \neq 0, \ j \geq 0.$$
Similarly, a vector $\eta \in H$ is $K$-cyclic if and only if
$$ \langle \eta, f_j \rangle \neq 0, \ \ j \geq 0.$$

\end{proposition}

\begin{proof} We know that $S \ker K^\ast \subset \ker K$, hence $\ker K^\ast = 0.$
In view of Corollary \ref{span} both systems of vectors $(f_j)_{j=0}^\infty$ and $(g_j)_{j=0}^\infty$
are complete in $H$.

Let $\xi \in H$ and $L = {\rm span} \{ (K^\ast)^n \xi, \ \ n \geq 0 \}.$ We already know that, if
$f_j \in L$, then $\langle \xi, g_j \rangle \neq 0.$ According to Corollary \ref{norm-conv}, the imposed factorization $K = AS$ implies the norm convergence of the expansion
$$ K^\ast = \sum_{j=0}^\infty \lambda_j \langle \cdot, g_j \rangle f_j,$$
and, via the $C^{(1)}$ functional calculus,
$$ h(K^\ast) = \sum_{j=0}^\infty h(\lambda_j) \langle \cdot, g_j \rangle f_j, \ \ h \in C^{(1)}(\R), \ h(0) =0.$$
We can assume that the function $h$ is a limit of polynomials vanishing at zero, therefore
$ h(K^\ast)\xi \in L$.

Assume $\langle \xi, g_j \rangle \neq 0$ and choose the function $h$ as above, to be non-zero at the point $\lambda_j$
and $h(\lambda_k) =0, k \neq j.$ We infer $f_j \in L$.

In conclusion, $L = H$ if and only if all skew Fourier coefficients $\langle \xi, g_j \rangle$ are non-zero.
The proof for a $K$-cyclic vector is similar.
\end{proof}

Without the injectivity assumption $\ker K = 0$, or the weaker statement $\ker K^\ast =0$, the preceding proof provides  cyclicity criteria for vectors
belonging to the closed ranges of $K^\ast$, respectively $K$. More precisely, we can relate the spanning property of a vector to the residues of the localized resolvent.
Although we deal here with symmetrizable operators, the analogy to the self-adjoint case is striking.

\begin{theorem}\label{resolvent} Let $K = S A \in {\mathcal L}(H)$ be a symmetrizable operator, where
$A$ is self-adjoint and $S>0$ compact. Denote by $\lambda_j$ the non-zero eigenvectors of $K^\ast$, with associated eigenfunctions
$f_j$ and biorthogonal system of $K$-eigenfunctions $g_j = Sf_j, j \geq 1.$ For every point $\lambda \in \sigma ({K^\ast}) \setminus \{ 0 \}$, let
$Q_\lambda = \sum_{j; \lambda_j = \lambda} \langle \cdot , g_j \rangle f_j.$

Let $f \in H$. The expansion of the meromorphic function
\begin{equation}\label{resolvent}
 (I - z K^\ast)^{-1} f = f + \sum_ {\lambda \in \sigma ({K^\ast}) \setminus \{ 0 \}} \frac{\lambda z Q_\lambda f}{1- z \lambda}
 \end{equation}
is convergent on the complement of the  set $\Sigma= \{ \frac{1}{\lambda_j}, \ j \geq 1 \},$ with residues
$$ {\rm Res}_\lambda (I - z K^\ast)^{-1} f = \frac{-Q_\lambda f}{\lambda}, \ \ \lambda \in \Sigma.$$

The function $ (I - z K^\ast)^{-1} f$ is entire if and only if $f \in \ker K^\ast.$ In that case  $ (I - z K^\ast)^{-1} f = f.$
\end{theorem}

\begin{proof} The convergence in Hilbert space $H$ of the expansion (\ref{resolvent}) follows from the existence of the $C^{(1)}$-functional calculus for the operator $K^\ast$.
According to Corollary \ref{span}, the closure of the range of $K$, that is $(\ker K^\ast)^\perp,$ is generated by the vectors $g_j, \ j \geq 1.$ Therefore $Q_\lambda f = 0$ for all
$\lambda \in \Sigma$, if and only if $K^\ast f = 0.$
\end{proof}

In general, the above statement remains true for all symmetrizable operators, except the norm convergence of the resolvent expansion. Specifically, the meromorphic function
 $ \sqrt{S}(I - z K^\ast)^{-1} f = (I - z C)^{-1} \sqrt{S}f$ involves the resolvent $ (I - z C)^{-1}$ of the self-adjoint operator $C$, and for the later the picture is well known.
 Speaking about the resolvent as meromorphic function, one cannot avoid mentioning the classical growth conditions for it and of the related Fredholm determinant.

 \begin{remark}\label{Fredholm}

 Assume, in the conditions of Theorem \ref{resolvent}, that the operator $K^\ast$ belongs to Schatten-von Neumann class ${\mathcal C}_p(H)$, where $p$ is a positive integer.
 Then Fredholm's regularized determinant (see Chapter IV in\cite{GK} )
\begin{equation}\label{eq: definition of Fredholm determinat} {\det}_p(z,K^\ast) = \det (I- zK^\ast) \exp [ {\rm tr} \sum_{k=1}^{p-1} \frac{ (z K^\ast)^k}{k}] \end{equation} 
collects in the zeros of an entire function all poles of the resolvent, so that
 $$ F(z,K^\ast) =  {\det}_p(z,K^\ast) (I-zK^\ast)^{-1} $$ is an operator valued entire function. Then
 both entire functions $ \det_p(z,K^\ast)$ and $F(z,K^\ast)$ have order $p$ of minimal type. For a proof, see \cite{Macaev}.

 \end{remark}

\section{Analysis of the Neumann-Poincar\'e operator}\label{analysisNP}

With all preliminaries in place, we can now turn to the Neumann-Poincar\'e operator $K^\ast$ associated to a bounded domain
$\Omega \subset \R^d$ with smooth boundary $\Gamma$. The minimal regularity of $\Gamma$ will be implicit in the main results below.
The underlying Hilbert space $H$ is $L^2(\Gamma)$ with respect to the surface area element induced by the embedding of $\Gamma$
in the Euclidean space $\R^d$.

According to the identification (\ref{PDO}) of the single and double layer potential operators $S$ and $K$ with specific pseudo-differential operators, one finds
$ \sqrt{S} H = H^{1/2}(\Gamma) $ and $H^{-1/2}(\Gamma)$ equal to the completion of $H$ with respect to the weaker norm $\langle S \cdot, \cdot \rangle.$
The operator $K$ turns out to be compact and  $S$-symmetric by Plemelj's identity (\ref{Plemelj}).

We maintain the notational conventions of the preceding sections denoting by $\lambda_j, \ j \geq 0,$ the non-zero eigenvalues of $K$. These are real values, and also coincide with the non-zero eigenvalues of $K^\ast$.
The corresponding eigenfunctions (belonging to $H$) are
$$ K^\ast f_j = \lambda_j f_j, \ \ j \geq 0,$$
respectively
$$ K g_j = \lambda_j g_j, \ \ g_j = Sf_j, \ \ j \geq 0,$$
normalized by
$$ \langle f_k, g_j \rangle = \delta_{kj}.$$

\subsection{Spectral resolution  of the Neumann-Poincar\'e operator}

The abstract factorization of a symmetrizable operator discussed in section \ref{factor} applies verbatim to the Neumann--Poincar\'e operator on a smooth, closed hypersurface in Euclidean space.
We exploit the consequences of this algebraic splitting, with some specific insights.

\begin{theorem}\label{split} Let $\Omega \subset \R^d, \ \ d\geq 2,$ be a bounded domain with smooth boundary $\Gamma$. The Neumann-Poincar\'e operator $K$ admits the factorization
$$ K = S A,$$
where $A$ is a pseudo-differential operator of order zero, for $d \geq 3$, and smoothing in case $d=2$. The principal symbol of $A$ is (\ref{symbolA}).
\end{theorem}

Carleman noticed in case $d=2$ the existence and importance of the splitting $K = SA$. See pg. 159 in \cite{Carleman}. From the higher stand of today's techniques, we can further exploit  the factorization, in any dimension.

\begin{corollary} In the conditions of Theorem \ref{split}, the spectral expansions
\begin{equation}
 K =  \sum_{n=0}^\infty \lambda_n \langle \cdot,  f_n \rangle g_n,
 \end{equation}
and
\begin{equation}\label{K*}
 K^\ast =   \sum_{n=0}^\infty \lambda_n \langle \cdot,  g_n \rangle f_n,
 \end{equation}
converge in operator norm.
\end{corollary}

This is word by word Corollary \ref{norm-conv}. Passing to strong operator topology convergence, one can say more:

\begin{corollary} a) Let $d=2$ in Theorem \ref{split}. Then all eigenfunctions of $K^\ast$ corresponding to non-zero eigenvalues are smooth functions $f_j \in C^{(\infty)}(\Gamma), \ \ j \geq 0,$
and the expansion
\begin{equation}\label{Mercer}
 K^\ast f =   \sum_{n=0}^\infty \lambda_n \langle f,  g_n \rangle f_n, \ \ f \in L^2(\Gamma),
 \end{equation}
converges in the Fr\'echet space topology of $C^{(\infty)}(\Gamma)$.

b) If $d \geq 3$, then the expansion (\ref{Mercer})  converges in the Sobolev space norm $H^1(\Gamma)$.
\end{corollary}

This is Krein's Theorem \ref{Krein}, in its turn analog of the improved convergence in the Hilbert-Schmidt framework. In dimensions greater or equal than three, $K^\ast$ is a pseudodifferential
operator of order $-1$, hence bounded from $L^2(\Gamma)$ to $H^1(\Gamma)$. A uniform bound in the expansion (\ref{Mercer}) would be a true analog of Mercer's Theorem. See for details
\cite{Krein}. Due to the compact embedding on the Sobolev scale $H^s(\Gamma) \subset H^t(\Gamma), \ s > t,$ we deduce from the above strong operator topology convergences uniform ones:

\begin{corollary} In the condition of Theorem \ref{split}:

a) If $d=2,$ then the convergence (\ref{K*}) holds in the operator norm \\
${\mathcal L}(L^2(\Gamma), H^s(\gamma))$ for all $s>0$.

b) For $d \geq 3$, the convergence  (\ref{K*}) holds in the operator norm \\
${\mathcal L}(L^2(\Gamma), H^s(\gamma))$ for all $0 \leq  s < 1$.
\end{corollary}

The improved spectral resolution of the Neumann-Poincar\'e operator is now within reach.

\begin{theorem}\label{resolution}
 Let $\Omega \subset \R^d, \ \ d\geq 2,$ be a bounded domain with smooth boundary $\Gamma$. For every non-zero eigenvalue $\lambda$ of $K^\ast$, consider the slanted spectral projection
 $Q_\lambda = \sum_{j; \lambda_j = \lambda} \langle \cdot, g_j \rangle f_j$. For every $f \in L^2(\Gamma)$ and $z \notin \sigma(K^\ast)$, the resolvent
  $$(I - z K^\ast)^{-1} f = f + \sum_ {\lambda \in \sigma ({K^\ast}) \setminus \{ 0 \}} \frac{\lambda z Q_\lambda f}{1- z \lambda}$$
  converges in any Sobolev space $H^s(\Gamma), \ \ s >0,$ for $d=2$, respectively in $H^1(\Gamma)$ in case $d \geq 3$. In both cases, the converges in uniform with respect to $f, \ \| f \| \leq 1.$
\end{theorem}

In the conditions of Theorem \ref{resolution}, note that, for every $g \in L^2({\Gamma})$,
\begin{equation}\label{Borel}
 \langle \frac{(I - z K^\ast)^{-1} f - f }{z}, g \rangle  = \sum_ {\lambda \in \sigma ({K^\ast}) \setminus \{ 0 \}} \frac{ \langle Q_\lambda f, g \rangle }{\frac{1}{\lambda} - z}
 \end{equation}
is a convergent Borel series of simple fractions, without $\sum \langle Q_\lambda f, g \rangle $ being absolutely convergent. However, the series $\sum \lambda \langle Q_\lambda f, g \rangle$
converges absolutely. Indeed,
$$ \sum_j | \lambda_j \langle f, g_j \rangle \langle f_j, g \rangle | =  \sum_j |  \langle f, g_j \rangle \langle K^\ast f_j, g \rangle | = $$ $$
\sum_j | \langle \sqrt{S} f, \sqrt{S} f_j \rangle \langle \sqrt{S} f_j, \sqrt{S} A g \rangle | \leq
\| \sqrt{S} f \| \| \sqrt{S} A g \|.$$
In conclusion, the Borel series (\ref{Borel}) is uniformly convergent on compact subsets of $\C \setminus \sigma(K^\ast)$, also uniformly with respect to $f, g$ in the unit ball of $L^2(\Gamma)$.

Assuming, in two real dimensions, a restricted smoothness of the boundary hypersurface $\Gamma \subset \R^2$, the above statements have to be adapted, taking into consideration the respective
smoothness of the integral kernel $K^\ast$. We leave the details to the reader.

\subsection{Growth of the resolvent}

A natural application of the smooth functional calculus appearing  in Theorem \ref{funct-calculus} is the analysis of the resolvent growth of the Neumann-Poincar\'e operator at a point of the spectrum.

\begin{theorem}\label{resolvent-growth}
 Let $\Omega \subset \R^d, \ \ d\geq 2,$ be a bounded domain with smooth boundary. Consider a non-zero eigenvalue $\alpha$ of the Neumann-Poincar\'e operator $K^\ast = AS$ associated to $\Omega$. Then
 \begin{equation}
 \frac{1}{|z - \alpha|} \leq \| (K^\ast - z)^{-1} \| \leq \frac{ 1+ \| A \| \|S \|}{|z-\alpha|},
 \end{equation}
 whenever  $0< |z-\alpha| < \delta$, where
  $ \delta = \inf_{\lambda \in \sigma(K) \setminus \{ \alpha \} } \frac{|\lambda-\alpha| }{2}.$
 \end{theorem}

 \begin{proof} The first inequality can be derived from Krein's estimate (\ref{lower-bound}). To prove the upper bound of the resolvent norm in the second inequality, we rely on the expansion
 $$ (K^\ast - z)^{-1} x = - \frac{x}{z} + \sum_{k=0}^\infty \frac{\lambda_n \langle x, g_k \rangle f_k}{\lambda - z}, \ x \in H,$$ valid
 for $z \notin \sigma(K^\ast).$ Going back to the proof of Theorem \ref{funct-calculus} we remark:
 $$ (K^\ast - z)^{-1} x + \frac{x}{z} = A \sqrt{S} \sum_{k=0}^\infty \frac{  \langle \sqrt{S} x, \sqrt{S} f_k \rangle \sqrt{S} f_k}{\lambda_k - z}.$$
 Since $(\sqrt{S}f_k)_{k=0}^\infty$ is an orthonormal system in $H$,
 $$  \| (K^\ast - z)^{-1} x + \frac{x}{z} \| \leq \|A \| \| \sqrt{S}\| \frac{\| \sqrt{S} x \|}{\inf_{\lambda \in \sigma(K^\ast)} |z-\lambda|}, \ \ x \in H.$$
 Consequently
 $$ \|  (K^\ast - z)^{-1} \| \leq \frac{1}{|z|} +  \frac{ \|A \| \| S\|}{\inf_{\lambda \in \sigma(K^\ast)} |z-\lambda|}.$$
 With the choice of $\delta$ in the statement $|z-\alpha| = \inf_{\lambda \in \sigma(K^\ast)} |z-\lambda|$. On the other hand,
 $$ \frac{1}{|z|} \leq \frac{1}{|\alpha| - |z-\alpha|} \leq \frac{1}{|z-\alpha|},$$
 because $|z-\alpha| < \frac{|\alpha|}{2}.$

 \end{proof}

\subsection{Microlocal analysis of the multiplicative factor}\label{subsec: Mamf}

In this section we focus on some qualitative features of the operator factor $A$.

In general, the essential spectrum of $\Psi$DO of order $0$ is the range of the principal symbol \cite{Adams}. Consequently
$$ \sigma_{ess}(A)
 = \mbox{Ran}(\sigma_0(A)) := \{ \sigma_0(A)(x, \xi)|\ (x, \xi) \in S^{\ast} \Gamma \} = $$ $$
 \bigcup_{x \in \Gamma} [ (d-1)H(x) - \kappa_+(x) ,
(d-1)H(x) - \kappa_-(x)],
$$
where
$$
\kappa_+(x) := \max_j (\kappa_1(x), \ldots, \kappa_{d-1}(x)), \quad \kappa_-(x):= \min_j (\kappa_1(x), \ldots, \kappa_{d-1}(x) ).
$$
Above $S^{\ast} \Gamma$ denotes the  cosphere bundle and $H(x)$ is the mean curvature at $x \in \Gamma$. In particular,  we have
$$\sigma_{\rm ess} (A) \subset [\kappa_-, \kappa_+].$$

\begin{proposition} Let $d \geq 3$ and assume that the hypersurface $\Gamma$ is strictly convex. Then the operator $A$ is essentially non-negative,
that is $A$ is a finite rank perturbation of a non-negative operator if the boundary $\Gamma$ is strictly convex.
\end{proposition}

As a corollary, one verifies that the spectrum of $K$ has only finitely many negative eigenvalues,
a result already derivable from the closed form of the principal symbol of $K$.

The notation $C \simeq D$ means, for two $\Psi$DO's, that $C-D$ is a $\Psi$DO of order $-1$.
One finds an invariant in the sense of defect measures:

\begin{theorem}[A stable observalble] \label{thm: defect measure} Let $d\geq 3$ and $\Gamma \subset \R^d$ be a smooth, closed hypersurface. Then
the operator $A$ is a stable observable under the geodesic flow on the cosphere bundle $S^{\ast} \Gamma$, that is,
$$
e^{-it \sqrt{-\triangle}} A e^{it \sqrt{-\triangle}} \simeq A  \quad \mbox{for all}\ t\in {\mathbf{R}}.
$$

\end{theorem}
\begin{proof}
It is known \cite{Taylor} that the operator $e^{-it \sqrt{-\triangle}} A e^{it \sqrt{-\triangle}}$ is a $\Psi$DO of order $0$. Then
\begin{align*}
\frac{d}{dt} (e^{-it \sqrt{-\triangle}} A e^{it \sqrt{-\triangle}}) &= -i (e^{-it \sqrt{-\triangle}} ( \sqrt{-\triangle} A - A \sqrt{-\triangle} )e^{it \sqrt{-\triangle}} \\
&= -i e^{-it \sqrt{-\triangle}} [ \sqrt{-\triangle}, A] e^{it \sqrt{-\triangle}}
\end{align*}
where $[ \sqrt{-\triangle}, A]$ is the Lie bracket. Since $\sigma_0(S^{-1}) = |\xi|_x$,  $\sqrt{-\triangle} \simeq S^{-1}$. It follows from Plemelj's principle that
\begin{align*}
[\sqrt{-\triangle}, A] & \simeq [S^{-1}, A] \\
& \simeq [S^{-1}, S^{-1} K] =S^{-1} (S^{-1} K) -  (S^{-1} K) S^{-1} \\
& = S^{-1} (K^{\ast} S^{-1} ) -  (S^{-1} K) S^{-1} \\
& = S^{-1} (K^{\ast} -K ) S^{-1}.
\end{align*}
It is proved in \cite{Miyanishi-Rozenblum} that $K-K^{\ast}$ is $\Psi$DO of order $-3$ or less.
Thus the order of $[S^{-1}, S^{-1} K]$ is $-1$ or less.
The finite time integration from $0$ to $t$ doesn't affect the principal symbol $A$.
\end{proof}
It follows that the geodesic Hamiltonian flow on the cosphere bundle
$\exp tX_H : S^{\ast}\Gamma \rightarrow S^{\ast} \Gamma$ doesn't change the function $\sigma_0(A)(x, \xi)$, that is,
\begin{align*}
 (\exp tX_H)^{\ast}  & \left[|\xi|_x^{-2} \left\{ \sum_{j=1}^{d-1} \kappa_{j}(x)  |\xi|_x^2 -L(\xi, \xi) \right\}\right] \\
 \equiv & \left[|\xi|_x^{-2} \left\{ \sum_{j=1}^{d-1} \kappa_{j}(x)  |\xi|_x^2 -L(\xi, \xi) \right\}\right] \quad \mod S^{-1}.
\end{align*}

\begin{remark} From Theorem \ref{thm: defect measure},
$$
e^{-it \Lambda} A e^{it \Lambda} \simeq e^{-it S^{-1}} A e^{it S^{-1}} \simeq e^{-it \sqrt{-\triangle}} A e^{it \sqrt{-\triangle}}\simeq A \quad (t \in {\mathbf{R}}).
$$
\end{remark}
\begin{corollary}[An invariant]
Let $\{ u_j \}_{j \in {\mathbf{N}}}$ be a sequence satisfying $\Vert u_j \Vert_{H^k}=1$ for all $j$ and $\displaystyle\lim_{j \rightarrow \infty} \Vert u_j \Vert_{H^l} =0$\ ($l <k$).
Then
$$
\lim_{j\rightarrow \infty} \langle u_j,  e^{-it \sqrt{-\triangle}} A e^{it \sqrt{-\triangle}} u_j \rangle_{H^k} = \lim_{j\rightarrow \infty} \langle u_j,  A u_j \rangle_{H^k}.
$$
\end{corollary}

The above formulas are closely related to the concentration of the Neumann--Poincar\'e eigenfunctions as investigated in \cite{ACL}.
The commutator identity $[A, S] = K- K^\ast$ has some other ramifications. For instance:

\begin{theorem} The factors $A$ and $S$ of the Neumann-Poincar\'e operator associated to a smooth domain in $\R^d, d \geq 2,$ commute, if and only if $\Omega$ is a ball.
\end{theorem}

This is Theorem 8 in \cite{KPS}. It is interesting to remark that a characterization of balls in Euclidean space was recently proved in \cite{GKLP} in terms of the commutativity of the Dirichlet to Neumann operator and the boundary Laplace operator.

\subsection{Cyclic vectors}

Assume the Neumann-Poincar\'e operator $K$ associated to a smooth domain $\Omega \subset \R^d$ is injective and has simple spectrum $\sigma(K) = \{ \lambda_j, \ \ j \geq 0 \}.$
In view of Corollary \ref{span}, both systems of eigenfunctions, $(f_j)$ for $K^\ast$ and $(g_j)$ for $K$, are complete in $H = L^2(\Gamma).$ By Proposition \ref{cyclic}, a vector $\xi \in H$ is $K^\ast$-cyclic if and only if
$\langle \xi, g_j \rangle \neq 0, \ \ j \geq 0.$ Finding a specific $K^\ast$-cyclic vector is essential for many applications. We indicate a possible path. Recall $E(x)$ denotes the fundamental solution of Laplace operator $\Delta$, in $\R^d$.

Let $z \notin \overline{\Omega}$ and
\begin{equation}\label{qz}
q_z(x):= a \cdot \nabla_z E(z-x), \quad x \in \partial\Omega,
\end{equation}
where $a$ is a constant, unit vector.
Define the transform
$$
F_n(z):= \langle q_z, Sf_n \rangle = a \cdot \nabla_z \int_{\partial\Omega}E(z-x) g_n(x) d\sigma(x).
$$
Thus $q_z$ is not orthogonal to all $g_n's$ if $z \notin \cup_{n=1}^\infty F_n^{-1}(0)$. Since $F_n$ is harmonic in $\R^d \setminus \overline{\Omega}$, and not identically zero, the level set $F_n^{-1}(0)$ is of measure zero, and so is $\cup_{n=1}^\infty F_n^{-1}(0)$.

It remains to show that for every $j, j\geq 0$, the vector $a$ is not orthogonal to $\nabla_z \int_{\partial\Omega}E(z-x) g_j(x) d\sigma(x)$, at least for some values of $z$. Indeed,
assume by contradiction that  $\nabla_z a \cdot S(g_j) = \nabla_z \int_{\partial\Omega}a \cdot E(z-x) g_j(x) d\sigma(x) =0$, for all $z \notin \overline{\Omega}$. After a rotation, one can assume
$a = (1,0,0,\ldots, 0).$ That is the harmonic function
$S(g_j)$ restricted to the complement of $\overline{\Omega}$ does not depend on the variable $x_1$. On the other hand $\lim_{x_1 \rightarrow \infty} S(g_j)(x_1, x_2, x_3, \ldots, x_n) = 0$ for all
fixed values of $x_2, x_3, \ldots, x_n$. Hence $S(g_j)$ vanishes on the complement of $\overline{\Omega}$. But that means that the boundary values in $H^1(\Gamma)$ of the harmonic function
$S(g_j)$ defined in $\Omega$ are zero. By the uniqueness of the solution to the Dirichlet problem with boundary data in $H^1(\Gamma)$, we infer $S(g_j) = 0$ and $g_j = 0$, a contradiction.

\begin{theorem}\label{thm:allmode}
Suppose that $\Gamma= \partial \Omega$ is smooth  and $\sigma (K) \setminus \{ 0 \}$ is simple. Then, for almost all $z \in \R^n \setminus \overline{\Omega}$, the function $q_z$ is cyclic for the range of the operator $K^\ast$.
\end{theorem}

Recall that the eigenfunctions $(f_j)$ form an orthonormal basis of the weak completion $H^{-1/2}(\Gamma)$ of $H$, hence the element $q_z$ is a fortiori cyclic with respect to $K^\ast \in {\mathcal L}(H^{-1/2}(\Gamma))$. The injectivity of $K$ condition in the statement can be relaxed, with the price of working only with vectors belonging to the closed range of $K^\ast$.

We proved in a recent companion note \cite{AKMP} that the non-zero spectrum of the Neumann-Poincar\'e operator on smooth boundaries is generically simple.
More precisely, genericity defined in this context with respect to the Hausdorff distance. In the same article we exploit finitely many  ``field sources" $q_z(x):= a \cdot \nabla_z E(z-x)$ as natural multicyclic vectors in case the spectral multiplicity of $K^\ast$ is bigger than one.

\subsection{The Poisson transform}

The factor $A$ appearing the Neumann--Poncar\'e operator provides the structure of Poisson's transform, on any bounded smooth domain. Directly from the jump formulae (\ref{jump}):

\begin{theorem}\label{Poisson}
 Let $\Omega \subset \R^d, \ d \geq 2,$ be a bounded domain with smooth boundary. Let $K^\ast = AS$ be the factorization of the Neumann-Poincar\'e operator.
For every $f \in L^2(\Gamma)$ one has
\begin{equation}\label{Poisson}
f = -S^i_{Af} - 2 D_f^i.
\end{equation}
\end{theorem}

Note that the right had side of the transform above is the boundary value (in a weak sense) of the harmonic function  $-S_{Af} - 2 D_f$ defined in $\Omega$. Similarly one can derive an exterior
Poisson transform:
$$ f = S^e_{Af} + 2 D_f^e, \ \ f \in L^2(\Gamma).$$

We specialize the formula to the point spectrum of the factor $A$.
First note that $\ker A = \ker K$, whence $\overline{ K^\ast L^2(\Gamma)} = \overline{ A L^2({\Gamma})}.$ Moreover, for an arbitrary $\mu \in \R$ one finds
$$\ker (A - \mu) = \ker S(A-\mu) = \ker (K - \mu S).$$
Thus, the point spectrum of $A$ coincides with the point spectrum of the linear pencil $K - \lambda S$, and it is real.

Assume $Ah = \mu h$ for some non-trivial function $h \in L^2(\Gamma)$. In view of the jump formulae \eqref{jump} we infer the following identities of a transmission problem flavor:
$$ (\mu S_h + 2 D_h)^i = - h,$$
and $$  (\mu S_h + 2 D_h)^e =  h.$$

In general, the spectral measure $E_A$ of the self-adjoint operator $A$ incorporate such pointwise identities into a continuum. Indeed, let $f \in L^2(\Gamma)$. The {\it Poisson transform}
$$ P_\Omega f  = - \int_\R (t S + 2D) E_A(d t) f$$
defines a pair of harmonic functions $u^i = P_\Omega f, u^e = - P_\Omega f,$ defined on $\Omega$, respectively $\R^d \setminus \overline{\Omega},$ with boundary values on $\Gamma$ equal to $f$.

With the distant aim an effective discretization of the Dirichlet problem, note that the operator $A$ acts on Poincar\'e's fundamental functions $(f_j, g_j)$ as follows:
$$ A g_j = \lambda _j f_j, \ \ j \geq 0.$$
Specifically, the harmonic function
$$u_j = -\lambda_j S^i_{f_j} - 2 D^i_{g_j}$$
has the boundary value $g_j$.

\begin{proposition}

 Let $\Omega \subset \R^d, \ d \geq 2,$ be a bounded domain with smooth boundary. Let $K^\ast = AS$ be the factorization of the Neumann-Poincar\'e operator, with spectral data
 $\lambda_j, K^\ast f_j = \lambda_j f_j, K g_j = \lambda_j g_j, \langle f_j, g_\ell\rangle = \delta_{j\ell}, \ j, \ell \geq 0.$
An element $f \in H^{1/2}(\Gamma) = \sqrt{S} L^2(\Gamma)$ admits the series decomposition
\begin{equation}\label{basis}
 f = h + \sum c_j g_j, \ \ \sum_j |c_j|^2 < \infty,
 \end{equation}
where $Kh =0$.
Then
$ A f = \sum_j c_j \lambda_j f_j,$ and
$$ u = - 2 D^i_h - \sum_j c_j [ \lambda_j S^i_{f_j} + 2 D^i_{g_j}]$$
solves the interior Dirichlet problem with boundary data $f$: $u|_\Gamma = f$.
\end{proposition}

\begin{proof}
To prove the decomposition (\ref{basis}) we refer to the compact self-adjoint operator $C \in {\mathcal L}(L^2(\Gamma))$, entering into the intertwining relation
$$ \sqrt{S} C = K \sqrt{S},$$
see Section 3.1. The eigenfunctions of $C$ corresponding to non-zero eigenvalues are $\sqrt{S}f_j, \ j \geq 0$. Together with an orthonormal basis of $\ker C$ they span
orthogonally Lebesgue space $L^2(\Gamma)$. Hence, every element $\phi \in L^2(\Gamma)$ can be written as
$$ \phi = \psi + \sum_j c_j \sqrt{S}f_j, \ \ \sum_j |c_j|^2 < \infty,$$
where $C \psi = 0.$ Then $f = \sqrt{S}\phi$ yields the decomposition (\ref{basis}).
\end{proof}

The complexity of the integral kernel of the factor $A$ appearing in the quotient of layer potentials on a smooth closed curve on the plane, can be assessed from the
following formula, directly derived from \eqref{Poisson}.

\begin{corollary} Let $\Gamma = \partial \Omega \subset \R^2$ be a closed, smooth curve surrounding a bounded domain $\Omega$. The integral kernel of the factor $A = K^\ast S^{-1}$ is related to the Green function $G_\Omega$ by the identity:
\begin{equation}\label{GreenFunction}
\frac{\partial G_\Omega (z, \zeta)}{\partial n_\zeta} = \frac{1}{2\pi} \int_\Gamma  \ln (z - u) A(u,\zeta) d\sigma(u) - 2 \frac{\partial \arg (z-\zeta)}{\partial \tau_\zeta}, \ z \in \Omega, \zeta \in \Gamma.
\end{equation}
\end{corollary}

Above $n_\zeta$ stands for the unit exterior normal at the point $\zeta$ on the curve, $\tau_\zeta$ is the unit tangent vector and $d\sigma$ is arc length.

\begin{proof} For the proof we recall that $A$ is a smoothing operator, hence it admits a continuous integral kernel. Moreover, in general Poisson's transform is implemented by the left handside integral kernel. Finally, the closed form expression for the layer potential transfroms $S$ and $K^\ast$ in two real variables are responsible for the right hand side terms.
\end{proof}

Simply connected domains with an explicit conformal mapping from the disk, or lemniscates provide examples with a tangible lest hand side appearing in (\ref{GreenFunction}).

\subsection{Soft comparison of the spectral asymptotics of layer potentials}

Assume the boundary $\Gamma$ of a bounded domain contained in $\R^d, \ d \geq 3$ is smooth and strictly convex. Then the factor $A$ entering into the Neumann-Poincar\'e operator
$K =  SA$ is essentially positive as we have seen at the beginning of \S \ref{subsec: Mamf}.
In line with the classics, we denote by $\mu_j(K)$ the {\it characteristic values} of the linear pencil $I - \mu K$, so that
we can arrange them (multiplicity included) in non-decreasing order:
$$ \mu_1(K) \leq \mu_2(K) \leq \ldots \leq \mu_n(K) \leq \mu_{n+1}(K) \leq \ldots,$$
satisfying $\lim_n \mu_n(K) = \infty$. The following bounds are then within reach.

We isolate from the preceding geometric setting a general framework applicable to abstract symmetrizable operators.

\begin{theorem}\label{thm: ratio estimate} Let $K = SA$ be a compact symmetrizable operator with $S > 0$ and the essential spectrum of $A$ contained in the interval
$[\kappa_{-}, \kappa_{+}]$, with $\kappa_{-} \geq  0.$
The characteristic values of $S$ and $K$ satisfy:
$$ \kappa_{-} \leq \liminf \frac{\mu_n (S)}{\mu_n(K)} \leq  \limsup \frac{\mu_n (S)}{\mu_n(K)}  \leq  \kappa_+.$$
\end{theorem}

\begin{proof} Let $\epsilon >0$. Since the essential spectrum of the self-adjoint operator $A$ is contained in the interval $[\kappa_-, \kappa_+]$ one can write
$A = A_0 + A_1$, where $A_1$ is a finite rank operator and the spectrum of $A_0$ is contained in $(\kappa_- - \frac{\epsilon}{2}, \kappa_+ + \frac{\epsilon}{2}).$
The spectra of $K, K^\ast$ and $C = \sqrt{S} A \sqrt{S}$ coincide, with the advantage of self-adjointness for $C$. Let $V_n, n \geq 2,$ denote the linear span of the eigenvectors corresponding to the characteristic values $\mu_1(S), \ldots, \mu_{n-1}(S)$. In virtue of the min-max principle,
$$ \frac{1}{\mu_n(K)} \leq \max_{x \perp V_n} \frac{ \langle \sqrt{S} A \sqrt{S} x, x \rangle}{\|x \|^2}.$$
Without loss of generality we can assume $n$ large enough so that $\mu_n(K) > 0$.
Since $ \langle A_1 \sqrt{S}x, \sqrt{S} x\rangle=0$ for $x \perp V_{n-1}$ if $n$ is sufficiently large, we have
$$ \langle \sqrt{S} A \sqrt{S} x, x \rangle = \langle A_0 \sqrt{S}x, \sqrt{S} x\rangle$$
for such $x$.
Since the spectrum of $A_0$ is contained in $(\kappa_- - \frac{\epsilon}{2}, \kappa_+ + \frac{\epsilon}{2})$, we have

$$ \frac{\langle Sx,x \rangle} {\|x\|^2} (\kappa_--\frac{\epsilon}{2}) \leq \frac{ \langle \sqrt{S} A \sqrt{S} x, x \rangle}{\|x \|^2} \leq \frac{\langle Sx,x \rangle} {\|x\|^2} (\kappa_+ +\frac{\epsilon}{2}).$$
Since
$$  \max_{x \perp V_n} \frac{ \langle S x, x \rangle}{\|x \|^2} = \frac{1}{\mu_n(S)},$$
we prove the inequality
$$  \frac{\mu_n (S)}{\mu_n(K)} \leq \kappa_+ + \epsilon,$$
for sufficiently large $n$.

If $\kappa_- = 0$ there is nothing to verify for the lower bound. Assume $\kappa_- >0$. By choosing optimal finite dimensional subspaces for the operator $C$, rather than $S$, one derives again from the min-max principle the inequality
$$  \kappa_{-} - \epsilon \leq \frac{\mu_n (S)}{\mu_n(K)},$$
for $n$ sufficiently large.
\end{proof}

The singular numbers of the layer potentials satisfy similar bounds. We state a comprehensive result, applicable to all symmetrizable operators.
 Recall that the singular numbers $\sigma_n(T)$ of a compact operator $T$ are the eigenvalues (necessarily non-negative) of its modulus $|T| = \sqrt{T^\ast T}$.
Note the identity
$$ K^\ast K = S A^2 S,$$
and repeat the proof to have obtain the following.
\begin{corollary}\label{cor: ratio estimate} Let $K = SA$ be a compact symmetrizable operator, with the essential spectrum of $|A|$ contained in the interval $[\kappa_-,\kappa_+].$ The singular numbers of $S$ and $K$ satisfy
$$ \kappa_- \leq \liminf \frac{\sigma_n(K)}{\sigma_n(S)} \leq \limsup \frac{\sigma_n(K)}{\sigma_n(S)} \leq \kappa_+.$$
\end{corollary}

In the case of layer potentials, more elaborate proofs of the above results are available, for instance derived from Birman-Solomyak Theorem, see \cite{Miyanishi:Weyl}. The advantage of the above universal framework is its applicability to any symmetrizable operator. Note that when applied to layer potentials, the lower and upper bounds in Theorem \ref{thm: ratio estimate} can be replaced by
$$ \tilde{\kappa}_{-} =  \min_{x \in \Gamma} [\sum_{j=1}^{d-1} \kappa_j(x) - \kappa_+(x)]=\min_{x \in \Gamma}[(d-1) H(x) - \kappa_+(x)],$$
respectively
$$ \tilde{\kappa}_{+} = \max_{x \in \Gamma}  [\sum_{j=1}^{d-1} \kappa_j(x) - \kappa_-(x)]=\max_{x \in \Gamma}[(d-1) H(x) - \kappa_-(x)]$$
provided $ d \geq 3$ and $ \tilde{\kappa}_{-} \geq 0.$

%%%%%%%%%%%%%%%%%%%%%%%%%%%%%%%%%%%%%%%%%%%%%%%%%%%%%%%%%
\subsection{Sharp comparison of the spectral asymptotics of layer potentials}

The present section complements the previous one, by establishing the existence of the limit $$ \displaystyle \lim_{n\to\infty} \frac{\mu_n (S)}{|\mu_n(K)|}=\lim_{n\to\infty} \frac{\sigma_n(K)}{\sigma_n(S)}.$$

The constant  $C_\Gamma$ turns out to be expressible in closed form in terms of classical differential geometric entities. As before, the bounded domain $\Omega \subset \R^3$ is assumed to possess a smooth boundary $\Gamma$.

Already we know that under some positivity conditions (such as the strict convexity of the underlying domain) the asymptotics of the quotients $\frac{\mu_n (S)}{\mu_n(K)}$ and $\frac{\sigma_n(K)}{\sigma_n(S)}$
can be estimated by $\kappa_{\pm}$.
In fact, we prove below more, namely that the two ratios converge in absolute value to a constant of a deep geometric nature. First we recall some terminology.
The {\it Willmore energy} $W(\Gamma)$  of $\Gamma = \partial\Omega$ is defined as the average of the squared mean curvature $H(x)$ on $\Gamma$:
\begin{equation}\label{eq: definition Cs Ck}
W(\Gamma) =\int_{\Gamma} H^2(x) dS_x.$$  The Euler characteristic of  $\Gamma$ is denoted $\chi(\Gamma).$ Let $$C_{K} =\left( \frac{3W(\Gamma) -2\pi \chi(\Gamma)}{32 \pi} \right)^{1/2}\ \mbox{and}\quad C_S =\left(\frac{\text{Area}(\Gamma)}{16{\pi}} \right)^{1/2}.
\end{equation}
It is worth mentioning that $C_K$ is scale--invariant  (more precisely invariant under M\"{o}bius transforms \cite{Blaschke, Wh}) and $C_S$ represents a homothetic ratio of $\Gamma$.

The eigenvalue asymptotic of $S$ is known to be
\begin{equation}\label{eq: asymptotic of single layer}
\mu_n(S)^{-1} = \sigma_n(S) = C_S n^{-1/2}+o(n^{-1/2})\ \mbox{ as}\ n \to \infty.
\end{equation}
We will mention on derivations of this formula later in this subsection. The spectral asymptotics for the Neumann-Poincar\'e operator was established in \cite{Miyanishi:Weyl}:
\begin{equation}\label{eq: asymptotic of NP}
|\lambda_n(K)| \sim \sigma_n(K) = C_K n^{-1/2} + o(n^{-1/2})\ \mbox{ as}\ n \to \infty.
\end{equation}
Note that for a strictly convex domain $\mu_n(K) = |\lambda_n(K)|^{-1}$ except finitely many values of $n$.
In general, the spectral asymptotics of a non-symmetric perturbation of a self-adjoint operator is masterly analyzed in \cite{MM}.
It is also emphasized that the above facts do not hold in two dimensions ($d=2$).
We refer the reader to \cite{Zo} for the spectral asymptotics of single layer potential operators and to
\cite{AKM, FKM, FJKM} for the NP operator, all references specialized to two dimensions.

As an immediate consequence, we obtain the following:
\begin{theorem}\label{sharp-ratio}
Let $\Omega \subset \R^3$ be a bounded domain with smooth boundary $\Gamma$. Then
\begin{equation}\label{eq: limiting ratio}
\lim_{n\to\infty} \frac{\mu_n (S)}{|\mu_n(K)|} =  \lim_{n\to\infty} \frac{\sigma_n (K)}{\sigma_n(S)} =C_{\Gamma}=C_{K} /C_{S}.
\end{equation}
\end{theorem}
Theorem \ref{sharp-ratio} together with Corollary \ref{cor: ratio estimate} leads us to novel geometric inequalities:
\begin{corollary}Under the same assumption of Theorem \ref{sharp-ratio},
\begin{equation}\label{eq: Willmore principal curvature}
\kappa_{-} \leq \left[ \frac{3W(\Gamma) -4\pi}{2 \mbox{Area}{(\Gamma)}} \right]^{1/2} \leq \kappa_{+}.
\end{equation}
\end{corollary}
The first inequality in \eqref{eq: Willmore principal curvature} can also be proved by elementary arguments. However, we do not know any other way to prove the second one. Geometric meaning of these inequalities is intriguing; it would be interesting to find one. The equalities in \eqref{eq: Willmore principal curvature} hold when $\Gamma$ is a sphere. It is also interesting to find out if the converse holds.

Hereafter until the end of this subsection, we focus on the details referring to \eqref{eq: asymptotic of single layer} in the three dimensional case ($d = 3$). The analysis in higher dimensions is very similar, with a  constant  $C_{S}$ depending on the dimension $d$.

The formula \eqref{eq: asymptotic of single layer} can be established in two ways.
One route relies on the symbol of $S$.
We noted in \S \ref{PDO} that $S \equiv (\sqrt{-\triangle})^{-1}$ modulo $\Psi DO$ of order $-2$.
As the asymptotic behavior of eigenvalues and singular numbers depends only on the principal symbol, we are led to consider the well charted square root of the Laplacian.
The first term in the eigenvalue asymptotics of $(\sqrt{-\triangle})^{-1}$ is well known:
$C_S n^{-1/2}+o(n^{-1/2})\ \mbox{as}\ n \to \infty.$ Thus we obtain \eqref{eq: asymptotic of single layer}.
We emphasize that the method of microlocal analysis \cite{AKSV, Hormander} allows us
to give the more precise estimation of the equation \eqref{eq: asymptotic of single layer}. Specifically, we state the following theorem.
\begin{theorem}\label{Weyl asymptotic of single layer potentials in three dimensions}
If the periodic geodesic flow on the cosphere bundle $S^*\Gamma$ has measure $0$,
then
\begin{equation}
\lambda_{n}(S_{\Gamma})= C_S  n^{-1/2}+o(n^{-1}).
\end{equation}
\end{theorem}

The second term of Weyl's law of single layer potentials is $0$ for generic surfaces, since the periodic geodesic flow on the cosphere bundle $S^{*} \Gamma$ has measure $0$.

A different proof of \eqref{eq: asymptotic of single layer} is derived from the relationship  between the ND (Neumann--Dirichlet) map and the single layer potential (see \eqref{eq: relationship between Dirichlet-Neumann and layer potentials}). In view of the non-self-adjoint perturbation analysis due to Marcus and Macaev \cite{MM}, the asymptotics \eqref{eq: asymptotic of single layer} is similar to the half of the asymptotics of the ND map, which is known to be (\cite{GP, GKLP}): 
\begin{equation}
\lambda_n(\Lambda^{-1})=2C_S  n^{-1/2} +O(n^{-1})
\end{equation}
as $n \to \infty$.

\subsection{Fredholm regularized determinants, zeta regularized determinants and the spectral zeta function }\label{sec: determinants and zeta}

As recalled in \S \ref{synthesis}, the Fredholm's regularized determinants \eqref{eq: definition of Fredholm determinat} 
can be defined for Schatten
class $p$ operators $T$, where $p$ is a positive integer.  The following identity
\begin{equation}\label{Contiuation of renormalized determinant}
\mbox{det}_{p+1}(I+T) =\mbox{det}_{p}(I+T) \exp[(-1)^p \mbox{Tr}(T^p)/p].
\end{equation}
is instrumental in passing from a Schatten class $p$  to $p+1$, see for a proof and details \cite{Simon}[Theorem 6.2]. The correction exponential term can be
computed via Lidiskii's theorem: $\mbox{Tr}(T^p) = \sum \lambda_j(T)^p =: \zeta_T(p)$.

Above $\zeta_T(p)$ is the so-called {\it spectral zeta function} and its domain can be extended to non-integer values
in case the compact operator $T$ has at most finitely many eigenvalues avoiding the open right half-plane. Specifically, one defines
$\lambda_j(T)^p = e^{p(\ln( \lambda_j(T)))}$ whenever $\Re \lambda_j (T) > 0$. In the exceptional case $\Re \lambda_j(T) \leq 0$ and $\lambda_j(T) \neq 0$ a branch of the logarithm is
chosen to make sense of $\lambda_j(T)^p$.

Single layer potential operators on smooth domains fit into this scheme. On special classes of domains, the Neumann-Poincar\'e operator $K^\ast$ is known to possess only finitely many negative eigenvalues. For instance  on prolate spheroids,  $K^\ast$ has only  postive eigenvalues \cite{AA}. Also we saw that on a strict convex domain $K^\ast$
has only finitely many negatiuve eigenvalues, see also \cite{Miyanishi-Rozenblum}.

Regarding the layer potentials on strictly convex domains as $\Psi$DOs, the profound methods of \cite{Shubin} [Chapter 2] apply, to effect that the function $\zeta_T(z)$
originally defined only for for $\Re z > 2$ can be continued to a meromorphic function in the entire complex $z$--plane with
at most simple poles. Warning: although $K^\ast$ is not self-adjoint, the meromorphic continuation process applies since the principal symbol of $K^\ast$ is real valued.

Instrumental in the meromorphic continuation of $\zeta_T(z)$ is the following general result having in the background the Mellin transform of a Dirichlet series.

\begin{theorem}\label{thm: moromorphic extention of spectral zeta}
Let $\lambda_j$ be a sequence of positive numbers satisfying $\lambda_j  \sim C j^{-1/2}$ as $j \to \infty$, where $C$ is a constant.
Assume the function
\begin{equation}
\phi(t) =\sum_{m=0}^\infty e^{-\lambda_m^{-1} t} \quad(t>0),
\end{equation}
has an asymptotic expansion of the form
\begin{equation}\label{eq: asymptotic expansion}
\phi(t) \sim \sum_{n=-2}^{\infty} a_n t^n \quad (t\rightarrow 0).
\end{equation}
Then the function $\Psi(z) = \sum_{m=0}^\infty \lambda_m^z$  admits a meromorphic continuation to the complex $z$--plane with simple poles of residue $2{a_{-2}}$ at $z=2$, $a_{-1}$ at $z=1,$ and no other singularities.
\end{theorem}
\begin{proof}
The function $\phi(t)$ is of rapid decay at infinity and has the asymptotic expansion \eqref{eq: asymptotic expansion} at
zero, so by the results in \cite{Zagier} we know that its Mellin transform $\mathcal{M}{\phi}(s)$ extends meromorphically
to all $s$, with simple poles of residue $a_n$ at most $s = -n\ (n = -2, -1)$. On the other hand, $\mathcal{M}{\phi}(s)={\mathbf \Gamma}(s)\Psi(s)$ (${\mathbf \Gamma}(s)$ is the gamma function) by a well known formula of Mellin transform:
$$
{\mathcal{M}}(e^{-\lambda^{-1} t})[s] = {\mathbf \Gamma}(s) \lambda^{s} \quad (\lambda>0).
$$
Warning: here we refer to the classical Gamma function and not to the boundary of a domain. On the other hand we know that ${\mathbf \Gamma}(s)$ has simple poles of residue $(-1)^n \cdot n!$ at
$s = -n \ (n = 0, 1, \ldots)$, has no other zeros or poles, and equals $1$ at $s = 1$ and $1/2$ at $s=2$  as desired.
\end{proof}

Returning to layer potentials on strictly convex domains in $\R^3$, it turns out that the associated spectral zeta functions have a simple pole at $z=2$.
The respective residues are given by the values $C_{S}, C_{K}$ in \eqref{eq: definition Cs Ck}

Given the specific spectral decomposition structure of the Neumann-Poincar\'e operator, the following expansion of its complex powers holds.

\begin{proposition}\label{prop: complex power of bdd operator} Let $\Omega$ be a strictly convex domain in $\R^d, \ d\geq 2.$ Let $\lambda_j, \ \ j \geq 0,$ denote the non-zero eigenvalues of the
associated Neumann-Poincar\'e operator $K^\ast$. Let $K^\ast f_j = \lambda_j f_j,$ and $Kg_j = \lambda_j g_j, \ j \geq 0,$ be the corresponding eigenvalues
normalized by $\langle f_j, g_k \rangle = \delta_{jk}, \ j,k \geq 0.$ Then the series
\begin{equation}\label{complex power}
(K^\ast)^z = \sum_{j=0}^\infty \lambda_j^z \langle \cdot, g_j \rangle f_j, \ \ \Re z > 1,
\end{equation}
converges in operator norm.
\end{proposition}

\begin{proof} The convergence is not affected by a finite number of negative eigenvalues. For every complex $z$ in the half-plane $\Re z > 1$,
the function $t^z = \phi(t) = e^{ z \ln t}$ is of class $C^{(1)}[0,1]$ and vanishes at $t=0$. Then Theorem \ref{funct-calculus} applies.
\end{proof}

\begin{corollary} In the conditions of Proposition \ref{prop: complex power of bdd operator}, the operator valued map
$ z \mapsto (K^\ast)^z$
is analytic in the half-plane $\Re z >1.$

\end{corollary}

\begin{example}
We consider the NP operator associated to the unit sphere in ${\mathbb{R}}^3$.
The NP eigenvalues are  $\frac{1}{2n+1}$ ($n\in {\mathbb N}_{\geq 0}$) with multiplicity $2n+1$ \cite{Poi1}.
The corresponding spectral zeta function $\zeta_{S^2} (z)$ can be written as
\begin{align}
\begin{split}\label{eq: spectral zeta NP S2}
\zeta_{S^2} (z) &:= \sum_n \left( \frac{1}{2n+1} \right)^z \times (2n+1)  \\
&= \sum_n \left( \frac{1}{2n+1} \right)^{z-1} =\left( 1- 2^{-z+1}\right) \zeta(z-1), \quad z>2
\end{split}
\end{align}
where $\zeta(z)$ denotes the Riemann zeta function. For real $z$, the function $\zeta(z)$  is holomorphic everywhere
except for a simple pole at $z = 1$ with residue $1$. Thus $\zeta_{S^2}(z)$ has a pole at $z=2$ with residue $1/2$.
In fact, for the case of the sphere, one can directly verify that the constant $C_K$ coincides with the residue of $\zeta_{S^2}(z)$:
$$C_K=\left( \frac{3W(S^2) -2\pi \chi(S^2)}{32 \pi} \right)^{1/2} =1/2.$$

An important alternative functional determinant, known as the {\it zeta renormalized determinant} is:
\begin{equation}
\det T = \exp(-\zeta_T'(0)).
\end{equation}
We know that $\zeta'(-1) =\frac{1}{12} - \ln G$ where $G$ is the so-called Glaisher--Kinkelin constant \cite{Finch} .
It follows from \eqref{eq: spectral zeta NP S2} that $\zeta'_{S^2}(0) = - (\log 2) 2^{-0+1} \zeta(0-1) + (1-2^{-0+1}) \zeta'(0-1) = \frac{\log 2}{6}-\frac{1}{12} +\ln G.$ Thus
\begin{equation}
\det K_{S^2} =\frac{ 2^{-1/6} e^{\frac{1}{12}}}{G}.
\end{equation}

It follows from the values $\zeta(-1)=-\frac{1}{12}$ and $\zeta(0) = -\frac{1}{2}$ that
\begin{equation}\label{values or Riemann zeta}
\zeta_{S^2} (0) =\frac{1}{12},\ \zeta_{S^2}(1) =0,\ \zeta_{S^2}(2)=\infty, \cdots.
\end{equation}
What is more explicit relationship among the above determinants and the spectral zeta function?
\end{example}

One expects that the zeta renormalized determinant $\det (K_{\partial\Omega})$ for general NP operators for convex surfaces $\partial\Omega$ in three dimensions
carries relevant geometric information.
Note however that the anomaly $(\det A) (\det B) \not= \det (AB)$ can occur for the zeta renormalized determinants \cite{KV}.

\begin{example}[An ellipsoid in $\R^3$]
Although finitely many negative NP eigenvalues exist, the zeta function is finitely different from the eta functions (See \S \ref{sec: eta functions}).
Consequently the corresponding zeta function is regular at $s=0$ and $\det K_{\partial\Omega}$ is well--defined.
\end{example}

\subsection{Eta function: a prologue to an index theory for symmetrizable operators}\label{sec: eta functions}
Let us consider a symmetrizable pseudo--diferential operator $A$ acting on sections of a vector bundle $E$ over $\Gamma = \partial\Omega$.
If the spectrum $\sigma(A)\backslash{0}$ consists only of eigenvalues,  the {\it eta function} $\eta_A(s)$ is generally denoted as
$$
\eta_A(s) := \sum_{\lambda_j\not=0} \frac{\sgn \lambda_j}{|\lambda_j|^s}
$$
where $s\in {\mathbb{C}}$ and the $\lambda_j$'s run over the eigenvalues of $A$. Then the series is absolutely convergent in the half--plane $\Re(s)>\frac{\dim \partial\Omega}{m}$,
$m$ being the order of $A$. If for instance the boundary is $S^2$, the eta function of the NP operator $K_{S^2}$ is
simply $\eta_{K_{S^2}}(s)=\zeta_{S^2}(-s)$, a series considered in the previous section.
We infer that the eta function for the Neumann-Poincar\'e operator on the sphere is regular at $s=0$ and $\eta(0)=\zeta_{S^2}(0)=\frac{1}{12}$. It is worth mentioning that one of the eta invariants for the NP operators is trivial in dimension 2, which is considered as the value $\eta(0)=1$ for every bounded $C^{1, \alpha}$ domain $\Omega \subset {\mathbb{R}}^2$. This fact is obtained also as a consequence of the existence of a harmonic conjugate.

Although the NP operator $K$ is of order $-1$ in three-- or higher-- dimensions, we can produce a first order elliptic operator as the inverse $K^{-1}$ on strictly convex surfaces.
Then it follows from the celebrated Atiyah--Patodi--Singer theorem \cite{APS} that the eta function is in fact regular at $s=0$,
for all strictly convex surfaces.

%{\bf Question for Yoshihisa: Shall we mention the eta in dimension 2? Where we have quite a few fine results about the fast decay of eigenvalues, plus their symmetry with respect to the origin.
%For instance, in the case of the ellipse, eta will be explicit. Or we stick to dimension three?}

We emphasize that the value $\eta(0)$ depends only on $\dim \mbox{Ker} K$ and the symmetrizable extensions of $K$ to the whole domain $\Omega$ (See e.g. \cite{Muller} and references therein for the recent progress.).
In such a framework  zeta function, renormalized determinants, eta invariants and the associated miscellaneous identities provide relationships between the NP spectrum and the boundary geometry.
As mentioned in Theorem \ref{sharp-ratio}, the boundary geometry is deeply linked to the spectrum of the double layer operator. We will resume these topics in a separate article.

%%%%%%%%%%%%%%%%%%%%%%%%%%%%%%%%%%%%%%%%%%%%%%%%%%%%%%%%%%%%%%%%%%%%%%%%%%%%%%%%%%%%%%%%%%

\end{document}